\documentclass[12pt]{amsart}
\usepackage{amssymb,amsfonts,amsthm,amsmath,amscd}
\usepackage{tikz}
\usepackage{enumerate}

\usetikzlibrary{calc, shapes, backgrounds }
\usepackage{pgfplots,pgfplotstable}
\pgfplotsset{width=7cm}

\theoremstyle{plain}
\newtheorem{theorem}{Theorem}

\newtheorem{proposition}{Proposition}
\newtheorem{lemma}{Lemma}
\newtheorem{corollary}{Corollary}
\newtheorem{conjecture}{Conjecture}

\theoremstyle{definition}

\newtheorem{definition}{Definition}

\newtheorem{example}{Example}

\newtheorem{algorithm}{Algorithm}
\newtheorem*{ack}{Acknowledgement}
\newtheorem{remark}{Remark}

\def\G{\mathcal{G}}
\def\H{\mathcal{H}}

\def\P{\mathcal{P}}
\def\N{\mathcal{N}}

\def\A{\mathcal{A}}
\def\B{\mathcal{B}}

\def\S{\mathcal{S}}

\def\mex{\operatorname{mex}}

\begin{document}

\title{The Sprague-Grundy function for some nearly disjunctive sums of Nim and Silver Dollar games}

\author{Graham Farr}
\address{Faculty of IT, Monash University, Clayton, Victoria 3800, Australia}
\email{graham.farr@monash.edu}

\author{Nhan Bao Ho}
\address{Department of Mathematics and Statistics, La Trobe University, Bundoora, Victoria 3086, Australia, Australia}
\email{nhan.ho@latrobe.edu.au, nhanbaoho@gmail.com}



\thanks{Nhan Bao Ho was a 2014 Endeavour Research Fellowship recipient.  He also
thanks Monash University for hosting the fellowship.}

\begin{abstract}
We introduce and analyse an extension of the disjunctive sum operation on some classical impartial games.  Whereas the
disjunctive sum describes positions formed from independent subpositions, our operation combines positions that are not
completely independent but interact only in a very restricted way.  We extend the games Nim and Silver Dollar, played by
moving counters along one-dimensional strips of cells, by joining several strips at their initial cell.  We prove that, in certain cases, computing the Sprague-Grundy function can be simplified to that of a simpler game with at most two tokens in each strip.  We give an algorithm that, for each Sprague-Grundy value $g$, computes the positions of two-token Star Nim whose Sprague-Grundy values are $g$.  We establish that the sequence of differences of entries of these positions is ultimately additively periodic.
\end{abstract}

\maketitle

\section{Introduction}   \label{Ss.intro}

In the theory of combinatorial games, the \textit{disjunctive sum} plays a central role.  This dates back to Bouton's
pioneering work on {\sc Nim}, which became the archetypal impartial game \cite{Bouton}. A disjunctive sum represents a game
that can be broken down into independent smaller games. Such a game can be efficiently analysed in a  divide-and-conquer
fashion; specifically, its Sprague-Grundy value may be obtained from the Sprague-Grundy values of its component games using
{\sc Nim}-addition (i.e., bitwise mod-2 addition). We briefly review the basics of combinatorial game theory in Section
\ref{Ss.SG}.

In many games, disjunctive sums arise commonly in actual play: a position that was initially large and complex may develop
into a set of totally separate self-contained  sub-positions. Indeed, analysis of this phenomenon in the partizan game {\sc
Go} was the initial spark that led Berlekamp, Conway and Guy to develop the theory of surreal numbers and combinatorial games
\cite[Prologue]{con}.  But it is probably even more common, in practice, for a game position to develop sub-positions that
are mostly, but not entirely, self-contained, so that there is still some small amount of interaction between them.

It is natural, then, to investigate impartial games obtained by combining smaller games that are ``nearly'' independent, but
not completely so.  This line of enquiry has analogies  in many areas of mathematics, for example in connectivity in graphs,
where small separating sets of vertices or edges in connected graphs enable efficient divide-and-conquer steps  in many
algorithms.

In this paper, we study one of the simplest possible ways of combining impartial games in a ``nearly independent'' way.  We
focus on a family of games that generalize {\sc Nim}. These are based on moving tokens along one-dimensional strips of
squares, where squares are numbered by nonnegative integers. A single {\sc Nim}-heap is represented by a single strip with a
single token.  The index of  the token represents the size of the heap.  The token may be moved to any lower-numbered square.
A general {\sc Nim} position is obtained by taking a disjunctive sum of {\sc  Nim}-heaps, in which all the strips are
disjoint.

We introduce {\sc Star Nim}, in which we take some number $m$ of strips and identify the 0-squares of each strip.  We can
picture the $m$ strips radiating out from their shared  square at position 0. The token on each strip may be in any position,
including 0, but we forbid a square from containing more than one token.  So, as soon as a token is moved  down some strip to
the 0-square, all the other tokens on the other strips are forever prevented from going there. We also look at Star versions
of the {\sc Silver Dollar} game.

\subsection{The Sprague-Grundy theory for the impartial games} \label{Ss.SG}
In a two-player combinatorial game, the players move alternately, following some set of rules for moves. There is no hidden
information and no element of luck. All games discussed in this paper are short and impartial. A {\it short} game has a
finite number of positions and each position can be visited once (no loop) and so the game terminates after a finite number
of moves. A game is {\it impartial} if the two players have the same options for moves from every position. The player who
makes the last move wins (normal convention). More comprehensive theory can be found in \cite{ww1, con}.


A position is an {\it $\N$-position} if the next player (the player about to move) can have a plan of moves to win and a {\it
$\P$-position} otherwise \cite{ww1}. The {\it terminal} position is the position without a legal move, and so is a
$\P$-position.


Introduced by Grundy \cite{Gru39} and Sprague \cite{Spr36, Spr37}, the {\it Sprague-Grundy value} $\G(x)$ of a position $x$
is defined recursively as follows: the terminal position has value $0$, and $\G(x) = n$ if for every $m$ such that $0 \leq m
< n$, one can move from $x$ to some $y$ such that $\G(y) = m$ and there is no move from $x$ to $z$ such that $\G(z) = n$.
Note that a position is a $\P$-position if and only if its Sprague-Grundy value is zero. The following lemma follows from the
definition of Sprague-Grundy values.

\begin{lemma} \label{g-positions} \cite{Ext}
For each $g$, the set $S_g$ of positions whose Sprague-Grundy value is $g$ satisfies the following conditions:
\begin{enumerate} \itemsep0em
\item [\rm{(1)}] there is no move between two distinct positions in $S_g$, and
\item [\rm{(2)}] from any position $p$ not in $\cup_{i=0}^gS_i$, there exists a move that terminates in $S_g$.
\end{enumerate}
\end{lemma}

The Sprague-Grundy function $\G$ plays an important role in the study of the {\it disjunctive sum} of games defined as
follows. Given two games $G$ and $H$, the two players alternately move, choosing either of the two games and moving in that
game. The play ends when there is no move available from either of the games. In this paper, a sum means disjunctive sum. The
following theorem give us a winning strategy for playing sums, and uses bitwise mod-2 addition, denoted by $\oplus$, of
binary representations of numbers.

\begin{theorem} [\cite{Gru39, Spr36, Spr37}] \label{Thm.SG}
The Sprague-Grundy value of the sum of two games $G$ and $H$ is the {\sc Nim}-sum  $\G(G) \oplus \G(H)$.
\end{theorem}



The game of {\sc Nim}, analyzed by  Bouton \cite{Bouton}, provides a typical example of sums. This game is played with a
finite number of piles of tokens. A move consists of choosing one pile and removing an arbitrary number of tokens from that
pile. A game with $n$ piles $a_1, a_2, \ldots, a_n$ is denoted by {\sc Nim}$(a_1, a_2, \ldots, a_n)$. {\sc Nim} is a sum of
multiple one-pile {\sc Nim}$(a_i)$. The Sprague-Grundy function of {\sc Nim} is the {\sc Nim}-sum of the sizes of the single
piles. The following lemma follows from Theorem \ref{Thm.SG}.

\begin{lemma} \label{SG.Nim} \cite{Bouton}
The Sprague-Grundy value of {\sc Nim}$(a_1, a_2, \ldots, a_n)$ is $\bigoplus_{i=1}^n a_i$.
\end{lemma}

\subsection{{\sc Star Silver Dollar}} \label{Ss.SSD}

The game of  {\sc Silver Dollar} \cite{con} is played with a finite strip of squares labeled from the left end by $0,1,2,
\ldots$ with at most one token on each square. The two players alternately move, choosing one token and moving it to an empty
square that is labeled with a smaller number without jumping over any of the other tokens. The game ends when there is no
move available and the player who makes the last move wins. 


For example, in Figure \ref{SD.position}, the allowed moves are: move the token on square 2 to square 1 or 0; or move 5 to 4
or 3; or move 8 to 7; or move 10 to 9. In this position, the token on square 6 cannot move.


\begin{figure}[ht]
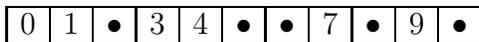

\begin{center}
\begin{tabular}{|c|c|c|c|c|c|c|c|c|c|c|}
\hline
0&1&$\bullet$&3 &4 &$\bullet$ &$\bullet$ &7 &$\bullet$ &9 &$\bullet$\\
\hline
\end{tabular}
\caption{A position in {\sc Silver Dollar} with bullets ``$\bullet$" representing tokens.}\label{SD.position}
\end{center}
\end{figure}

An analysis of Sprague-Grundy values for {\sc Silver Dollar} will be given in section \ref{S.SD}.

We generalize the game of {\sc Silver Dollar} to multiple strips. Given $m$ {\sc Silver Dollar}s, we overlap the zero squares
of these games, as illustrated in Figure \ref{abs}. A move consists of choosing an arbitrary token and moving it to some
smaller empty square on the same strip without jumping over any token. The next player to move can choose a different strip.
We call this game {\sc Star Silver Dollar}.



\begin{figure}[ht]
\begin{center}

\def\a{0.3}   

\def\TeeThree#1{\def\b{3}
\begin{scope}[rotate=#1]
    \draw [thick](0,\a) -- (\b/1.5+\a/4,\a) -- (\b/1.5+\a/3,-\a) -- (0,-\a);
   {\foreach \i in {1,2,3} \draw [thick] (\i/1.5+\a/4,\a) -- (\i/1.5+\a/4,-\a);}
    \draw [thick](0,0) circle (\a);
   {\foreach \i in {1}  \node at (\i/1.6-\a/4,0) {$\i$};}
   {\foreach \i in {2}  \node at (\i/1.65-\a/4,0) {$\i$};}
   {\foreach \i in {3}  \node at (\i/1.65-\a/4,0) {$\bullet$};}
  \end{scope}
}
\def\TeeFour#1{\def\b{4}
\begin{scope}[rotate=#1]
    \draw [thick](0,\a) -- (\b/1.5+\a/4,\a) -- (\b/1.5+\a/4,-\a) -- (0,-\a);
   {\foreach \i in {1,2,...,\b} \draw [thick] (\i/1.5+\a/4,\a) -- (\i/1.5+\a/4,-\a);}
    \draw [thick](0,0) circle (\a);
   {\foreach \i in {2}  \node at (\i/1.65-\a/4,0) {$\i$};}
   {\foreach \i in {3}  \node at (\i/1.63-\a/4,0) {$\i$};}
   {\foreach \i in {1}  \node at (\i/1.55-\a/4,0) {$\bullet$};}
   {\foreach \i in {4}  \node at (\i/1.6-\a/4,0) {$\bullet$};}
\end{scope}
}
\def\TeeFive#1{\def\b{5}
\begin{scope}[rotate=#1]
  \draw [thick](0,\a) -- (\b/1.5+\a/4,\a) -- (\b/1.5+\a/4,-\a) -- (0,-\a);
  {\foreach \i in {1,2,...,\b} \draw [thick] (\i/1.5+\a/4,\a) -- (\i/1.5+\a/4,-\a);}
  \draw [thick](0,0) circle (\a);
  {\foreach \i in {1}  \node at (\i/1.5-\a/2,0) {$\i$};}
  {\foreach \i in {3,4}  \node at (\i/1.6-\a/2,0) {$\i$};}
  {\foreach \i in {2}  \node at (\i/1.6-\a/2,0) {$\bullet$};}
  {\foreach \i in {5}  \node at (\i/1.55-\a/2,0) {$\bullet$};}
\end{scope}
}
\begin{tikzpicture}[scale = .8, font=\small]
\node at (0,0){$0$};
\TeeFive{0}
\TeeThree{120}
\TeeFour{240}
\end{tikzpicture}
\caption{The {\sc Star Silver Dollar} position ([2,5], [3],[1,4]).}\label{abs}
\end{center}
\end{figure}
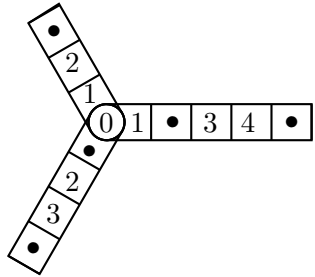


\begin{remark} \label{R.0-SSD}
We can assume that our {\sc Star Silver Dollar} does not have a token on square 0. If there is token on this square, the game
is the disjunctive sum of separate {\sc Silver Dollar}s and it can be solved as shown in Remark \ref{R.0-SD}. \qed
\end{remark}


We use $[a_1, a_2, \ldots, a_k]$ to represent a {\sc Silver Dollar} with tokens in the squares $a_1, a_2, \ldots, a_k$ with
$a_1 < a_2 < \cdots < a_k$. For example, the position in Figure \ref{SD.position} is represented by $[2,5,6,8,10]$. The {\sc
Star Silver Dollar} formed by two {\sc Silver Dollar}s is represented by $([a_1, a_2, \ldots, a_k], [b_1, b_2, \ldots, b_l])$
and so on.  We use {\sc $m$-Star Silver Dollar} when emphasizing that the game has $m$ strips.


{\sc Star Silver Dollar} is called {\sc Star Nim} if each  strip has only one token. The game {\sc $m$-Star Nim} is {\sc Star
Nim} with $m$ strips.


Several games played on strips have been studied, including {\sc Hexad} \cite{CS.86}, {\sc Welter} \cite{Welter1, Welter2}
and its variants \cite{Ho, k-Welter1}. One may be interested in studying the combination of these games using the idea of
{\sc Star Silver Dollar}.


We next recall the Sprague-Grundy function for {\sc Silver Dollar} for the reader's convenience.

\subsection{An analysis of {\sc Silver Dollar}} \label{S.SD}
The Spague-Grundy function for {\sc Silver Dollar} is analyzed in \cite[Chapter 11]{con} as follows. For each {\sc Silver
Dollar} position $x = [x_1, x_2, \ldots, x_n]$,
\begin{align*}
\G(x) =
\begin{cases}
\bigoplus_{i = 1}^{n/2}(x_{2i}-x_{2i-1}-1),                                        & \text{ if $n$ is even};  \\
x_1 \oplus \bigoplus_{i = 1}^{{(n-1)}/2} (x_{2i+1}-x_{2i}-1),       & \text{ otherwise}.
\end{cases}
\end{align*}
For example, the Spague-Grundy value for the {\sc Silver Dollar} position $[3,5,8,12,19]$ is $3 \oplus (8-5-1) \oplus
(19-12-1) =  7$. Note that $x_{j+1} - x_j- 1$ is the number of empty squares between the two squares $x_j$ and $x_{j+1}$.

Note that {\sc Silver Dollar} can be considered as the game of {\sc Nim} in which each term in the formula for Sprague-Grundy
function above is a pile in {\sc Nim}. The only difference with {\sc Nim} is that {\sc Silver Dollar} sometimes increases one
pile when reducing another pile. This is because moving a token $x_i$ reduces the gap between that token and token $x_{i-1}$
on its left but increases the gap between $x_i$ and the token $x_{i+1}$ on its right. However, this increase does not repeat
forever and also does not affect the winning strategy.


\begin{remark} \label{R.0-SD}
It can be seen that {\sc Star Silver Dollar} with token in square 0 is the sum of separate {\sc Silver Dollar} and so it is
solved by using the {\sc Nim}-sum of Sprague-Grundy values and the analysis of {\sc Silver Dollar} given just above.
\qed \end{remark}

For example, the {\sc Star Silver Dollar} in Figure \ref{SSD-0} is exactly the sum of three separate {\sc Silver Dollar}:
$[1,4]$, $[2]$, and $[2]$.

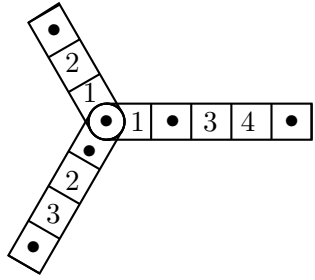
\begin{figure}[ht]
\begin{center}

\def\a{0.3}   

\def\TeeThree#1{\def\b{3}
\begin{scope}[rotate=#1]
    \draw [thick](0,\a) -- (\b/1.5+\a/4,\a) -- (\b/1.5+\a/3,-\a) -- (0,-\a);
   {\foreach \i in {1,2,3} \draw [thick] (\i/1.5+\a/4,\a) -- (\i/1.5+\a/4,-\a);}
    \draw [thick](0,0) circle (\a);
   {\foreach \i in {1}  \node at (\i/1.6-\a/4,0) {$\i$};}
   {\foreach \i in {2}  \node at (\i/1.65-\a/4,0) {$\i$};}
   {\foreach \i in {3}  \node at (\i/1.65-\a/4,0) {$\bullet$};}
  \end{scope}
}
\def\TeeFour#1{\def\b{4}
\begin{scope}[rotate=#1]
    \draw [thick](0,\a) -- (\b/1.5+\a/4,\a) -- (\b/1.5+\a/4,-\a) -- (0,-\a);
   {\foreach \i in {1,2,...,\b} \draw [thick] (\i/1.5+\a/4,\a) -- (\i/1.5+\a/4,-\a);}
    \draw [thick](0,0) circle (\a);
   {\foreach \i in {2}  \node at (\i/1.65-\a/4,0) {$\i$};}
   {\foreach \i in {3}  \node at (\i/1.63-\a/4,0) {$\i$};}
   {\foreach \i in {1}  \node at (\i/1.55-\a/4,0) {$\bullet$};}
   {\foreach \i in {4}  \node at (\i/1.6-\a/4,0) {$\bullet$};}
\end{scope}
}
\def\TeeFive#1{\def\b{5}
\begin{scope}[rotate=#1]
  \draw [thick](0,\a) -- (\b/1.5+\a/4,\a) -- (\b/1.5+\a/4,-\a) -- (0,-\a);
  {\foreach \i in {1,2,...,\b} \draw [thick] (\i/1.5+\a/4,\a) -- (\i/1.5+\a/4,-\a);}
  \draw [thick](0,0) circle (\a);
  {\foreach \i in {1}  \node at (\i/1.5-\a/2,0) {$\i$};}
  {\foreach \i in {3,4}  \node at (\i/1.6-\a/2,0) {$\i$};}
  {\foreach \i in {2}  \node at (\i/1.6-\a/2,0) {$\bullet$};}
  {\foreach \i in {5}  \node at (\i/1.55-\a/2,0) {$\bullet$};}
\end{scope}
}
\begin{tikzpicture}[scale = .8, font=\small]
\node at (0,0){$\bullet$};
\TeeFive{0}
\TeeThree{120}
\TeeFour{240}
\end{tikzpicture}
\caption{The {\sc Star Silver Dollar} with token on square zero.}\label{SSD-0}
\end{center}
\end{figure}

The outline of paper is as follows. In the next section, we show that the Sprague-Grundy function for {\sc Star Silver
Dollar} can be simplified to that for a simpler one where each strip has at most two tokens. In Section \ref{Ss.SN.P}, we
characterize the $\P$-positions for {\sc $m$-Star Nim} with $m \leq 4$. In Section \ref{Ss.SN.propety} we analyze the
Sprague-Grundy function of {\sc 2-Star Nim}. We give an algorithm that produces positions whose Sprague-Grundy values are
$g$.
We prove a periodicity property of the sequence of positions whose Sprague-Grundy values are $g$. We then prove an additive
periodicity property of the Sprague-Grundy function. In Section \ref{Ss.future}, we discuss some research direction for
further study. We also include some code written on Maple in Appendix \ref{A} for the reader's convenience.

\section{The Sprague-Grundy function for {\sc Star Silver Dollar}}   \label{Ss.SSD}

We show that {\sc Star Silver Dollar} can be described as a disjunctive sum of a simpler {\sc Star Silver Dollar} with at
most two tokens on each strip and separable {\sc Silver Dollar}s (Theorem \ref{SSD.MT}). We then examine the cases for the
simplified {\sc Star Silver Dollar}.
When each strip has an odd number of tokens, this simplified game is derived from {\sc Star Nim} (Corollary \ref{SSD.odd}).
When each strip has an even number of tokens, we establish a formula for the Sprague-Grundy function (Corollary
\ref{SSD.even}).


\begin{definition} \label{Head.Tail}
Given $m$ {\sc Silver Dollar}s $S^i$ with $1 \leq i \leq m$, for each $i$,
let $S^i_H$ be the {\sc Silver Dollar} containing exactly the smallest token (resp.~the two smallest tokens) of $S^i$ if
$S^i$ has an odd (resp. even) number of tokens, and
$S^i_T$ be the {\sc Silver Dollar} containing the remaining tokens of $S^i$. \qed
\end{definition}

In other words, if $S^i = [a^i_1,a^i_2,\ldots,a^i_{k_i}]$ then
$S^i_H = [a^i_1]$ and $S^i_T = [a^i_2,\ldots,a^i_{k_i}]$ if $k_i$ is odd and
$S^i_H = [a^i_1,a^i_2]$ and $S^i_T = [a^i_3,\ldots,a^i_{k_i}]$ otherwise.
Here $H$ stands for Head and $T$ stands for Tail.
Note that $S^i_H$ is not empty in the even case ($k_i$ is even).

\begin{example}
We give an example corresponding to the case where $S^i$ has an odd (resp.~even) number of tokens in Figure \ref{E.ST.odd}
(resp.~Figure \ref{E.ST.even}).
\end{example}

\begin{figure}[ht]
\quad \quad \quad \quad \quad \quad \quad $S^i \ $ =
\begin{tabular}{|c|c|c|c|c|c|c|c|c|c|c|}
\hline
0&1&$\bullet$&3 &4 &$\bullet$ &$\bullet$ &7 &$\bullet$ &9 &$\bullet$\\
\hline
\end{tabular}

\quad \quad \quad \quad \quad \quad \quad $S^i_H $ =
\begin{tabular}{|c|c|c|}
\hline
0&1&$\bullet$\\
\hline
\end{tabular}

\quad \quad \quad \quad \quad \quad \quad $S^i_T $ =
\begin{tabular}{|c|c|c|c|c|c|c|c|c|c|c|}
\hline
0&1&2&3 &4 &$\bullet$ &$\bullet$ &7 &$\bullet$ &9 &$\bullet$\\
\hline
\end{tabular}
\caption{An example when $S^i$ has an odd number of tokens.} \label{E.ST.odd}
\end{figure}

\begin{figure}[ht]
\quad \quad \quad \quad \quad \quad \quad $S^i \ $ =
\begin{tabular}{|c|c|c|c|c|c|c|c|c|}
\hline
0&1&$\bullet$&3 &4 &$\bullet$ &$\bullet$ &7 &$\bullet$\\
\hline
\end{tabular}

\quad \quad \quad \quad \quad \quad \quad $S^i_H $ =
\begin{tabular}{|c|c|c|c|c|c|}
\hline
0&1&$\bullet$&3 &4 &$\bullet$\\
\hline
\end{tabular}

\quad \quad \quad \quad \quad \quad \quad $S^i_T $ =
\begin{tabular}{|c|c|c|c|c|c|c|c|c|c|c|}
\hline
0&1&2&3 &4 &5 &$\bullet$ &7 &$\bullet$\\
\hline
\end{tabular}
\caption{An example when $S^i$ has an even number of tokens.}  \label{E.ST.even}
\end{figure}



\begin{theorem} \label{SSD.MT}
The game of {\sc Star Silver Dollar} $(S^1, S^2, \ldots, S^m)$ is exactly the disjunctive sum of new {\sc Star Silver Dollar}
$(S^1_H, S^2_H, \ldots, S^m_H)$ and new  {\sc Silver Dollar}s $S^i_T$. In other words,
\[ \G(S^1,S^2,\ldots,S^m) =
\underset{\text{new {\sc Star Silver Dollar}}}{\underbrace{\G(S^1_H, S^2_H, \ldots, S^m_H)}}
\oplus
\underset{\text{new {\sc Silver Dollar}s $S^i_T$}}
{\underbrace{
\G(S^1_T) \oplus \G(S^2_T) \oplus \cdots \oplus \G(S^m_T)
}}.\]
\end{theorem}

\begin{remark}
Note that {\sc Nim}-sum can be computed in linear time and so each $\G(S^i_T)$ can be computed in polynomial time using the
Sprague-Grundy function for {\sc Silver Dollar} (Section \ref{S.SD}). Therefore, the two functions $\G(S^1,S^2,\ldots,S^m)$
and $\G(S^1_H, S^2_H, \ldots, S^m_H)$ are polynomial-time computable.
\qed \end{remark}


\begin{example}
Consider {\sc Star Silver Dollar} $([2],[2,4,7],[1,4,6,10,12,17])$. We have $S^1 = S^1_H = [2]$, $S^2_H = [2]$, $S^2_T =
[4,7]$, $S^3_H = [1,4]$, $S^3_T = [6,10,12,17]$. Then
\begin{align*}
&\G([2],[2,4,7],[1,4,6,10,12,17])  \\
&= \G([2],[2],[1,4]) \oplus \G([4,7]) \oplus \G([6,10,12,17]) \\
&= \G([2],[2],[1,4]) \oplus(7-4-1) \oplus((10-6-1) \oplus (17-12-1)). \\
&= \G([2],[2],[1,4]) \oplus 5.
\end{align*}
\end{example}


\begin{proof} [Proof of Theorem \ref{SSD.MT}]

For each {\sc Star Silver Dollar} position $p = (S^1,S^2,\ldots,S^m)$, set
\[f(p) = \G(S^1_H, S^2_H, \ldots, S^m_H) \oplus \G(S^1_T) \oplus \G(S^2_T) \oplus \cdots \oplus \G(S^m_T).\]
We need to verify two facts:
\begin{enumerate} \itemsep0em
\item $f(p) \neq f(q)$ if there exists a move from $p$ to $q$, and
\item for any $g < f(p)$, there exists a move from $p$ to $q$ such that $f(q) = g$.
\end{enumerate}
It then follows that $f = \G$.

\item (1)
    Suppose one moves from $p$ to some $q$ by moving a token in some $S^i$. Without loss of generality, we can assume that $i
    = 1$. Denote by ${S^1}'$ the new {\sc Silver Dollar} obtained from $S^1$ by this move. Then, $q = ({S^1}', S^2, \ldots,
    S^m)$ and
    \[f(q) = \underset{\text{to be compared with $f(p)$}}{\underbrace{\G({S_H^1}', S^2_H, \ldots, S^m_H) \oplus
    \G({S_T^1}')}} \oplus \G(S^2_T) \oplus \cdots \oplus \G(S^m_T)\]
    in which either ${S_H^1}' = S^1_H$ and the move $S^1_T \to {S_T^1}'$ is available in $p$ or ${S_T^1}' = S^1_T$ and the
    move $S^1_H \to {S_H^1}'$ is available  in $p$. In both cases, $\G({S_H^1}', S^2_H, \ldots, S^m_H) \oplus \G({S_T^1}')
    \neq \G(S^1_H, S^2_H, \ldots, S^m_H) \oplus \G(S^1_T)$ and so $f(q) \neq f(p)$ since $\G(S^2_T) \oplus \cdots \oplus
    \G(S^m_T)$ is a constant.

    Note that this argument also holds for the case when the zero square in the position $p$ is occupied. In fact, in this
    case, the game is the disjunctive sum of separate {\sc Silver Dollar}s and the move from $p$ to $q$ is exactly a move in
    a single {\sc Silver Dollar} which changes the value of that game while keeping the values of other {\sc Silver Dollar}s
    unchanged.

\item (2) Let $g < f(p)$. Recall that each $S^i_T$ is a {\sc Silver Dollar} with an even number of tokens $(t^i_1,t^i_2,
    \ldots, t^i_{2l_i})$. The solution for {\sc Silver Dollar} gives $\G(S^i_T) = (t^i_2 - t^i_1 - 1) \oplus \cdots \oplus
    (t^i_{2l_i} - t^i_{2l_i-1} - 1)$. Expanding all $\G(S^i_T)$ gives
    \[f(p) = \G(S^1_H, S^2_H, \ldots, S^m_H)
    \oplus
    \bigg(
    \underset{\text{multiple one-pile {\sc Nim}$(t^i_l - t^i_{l-1} - 1)$}}{
    \underbrace{
    \bigoplus_{i \leq m} \big((t^i_2 - t^i_1 - 1) \oplus \cdots \oplus (t^i_{2l_i} - t^i_{2l_i-1} - 1)\big)
    }
    }
    \bigg).\]
By SG theory of sums, the right-hand side is the Sprague-Grundy function of the sum of {\sc Star Silver Dollar} game $(S^1_H,
S^2_H, \ldots, S^m_H)$ and multiple one-pile {\sc Nim}$(t^i_l - t^i_{l-1} - 1)$. Denote this sum by $G$. Then $\G(G) = f(p) >
g$.

By the definition of $\G$, there exists one move from $G$ to some $G'$ such that $\G(G') = g$. There are two possibilities
for this move: either \rm{(i)} moving some token in $(S^1_H, S^2_H, \ldots, S^m_H)$ or \rm{(ii)} reducing some {\sc
Nim}$(t^{i_0}_{l_0} - t^{i_0}_{l_0 - 1} - 1)$.

In the former case \rm{(i)}, without loss of generality, we can assume that this move affects some token in $S^1_H$,
resulting in ${S_H^1}'$. Then
 \[g = \G(G') =\G({S_H^1}', S^2_H, \ldots, S^m_H) \oplus \bigg( \bigoplus_{i \leq m} \big((t^i_2 - t^i_1 - 1) \oplus \cdots
 \oplus (t^i_{2l_i} - t^i_{2l_i-1} - 1)\big) \bigg).\]
Note that this move does not affect any {\sc Nim}$(t^i_{l'} - t^i_{l'-1} - 1)$. Consider the position $q$ obtained from $p$
by this move in $S^1_H$. Note that this move does not affect any $S^i_T$. We have
\[f(q) = \G({S_H^1}', S^2_H, \ldots, S^m_H) \oplus \big(\bigoplus_{i \leq m} \G(S^i_T)\big).\]
Substituting the expansion of $\G(S^i_T)$ into the right-hand side gives $f(q) = g$.

In the latter case \rm{(ii)}, assume that $s$ tokens have been removed from some {\sc Nim}$(t^{i_0}_{l_0} - t^{i_0}_{l_0 - 1}
- 1)$. Without loss of generality, assume that $i_0 = 1$. Then
\begin{align*}
 g =~ & \G(G') = \\
     & \G(S^1_H, S^2_H, \ldots, S^m_H) \\
     & \oplus \big( (t^1_2 - t^1_1 - 1) \oplus \cdots \oplus (t^{1}_{l_0} - t^{1}_{l_0 - 1} - 1 - s) \oplus \cdots \oplus
     (t^1_{2l_1} - t^1_{2l_1-1} - 1) \big) \\
     & \oplus \bigg( \bigoplus_{2 \leq i \leq m} \big((t^i_2 - t^i_1 - 1) \oplus \cdots \oplus (t^i_{2l_i} - t^i_{2l_i-1} -
     1)\big) \bigg).
\end{align*}
Consider the position $q$ obtained from $p$ by moving the token in the square $t^{1}_{l_0}$ to the square $t^{1}_{l_0} - s$
in $S^{1}_T$, resulting in $S'^{1}_T$. This move is available since $s \leq t^{1}_{l_0} - t^{1}_{l_0 - 1} - 1$ and so
$t^{1}_{l_0} - s > t^{1}_{l_0-1}$. Moreover, this move does not affect $S^{1}_H$. By the definition of $f$, we have
\begin{align*}
f(q) = \G(S^1_H, S^2_H, \ldots, S^m_H) \oplus \G({S_T^1}') \oplus \big(\bigoplus_{2 \leq i \leq m}\G(S^i_T)\big).
\end{align*}
Substituting the expansions of $\G({S_T^1}')$ and $\G(S^i_T)$ into the right-hand side gives $f(q) = g$.
\end{proof}


We consider the two special cases of Theorem \ref{SSD.MT} in which all strips have the same parity of numbers of tokens. The
following is immediate from Theorem \ref{SSD.MT}.


\begin{corollary} \label{SSD.odd}
If all strips have odd numbers of tokens, the game of {\sc Star Silver Dollar} is the sum of one {\sc Star Nim}, with the
lowest position token on each strip and separate strips with the remaining tokens.
\end{corollary}

For example, direct calculation shows that $\G([2],[2],[1]) = 3$ and so we have
\begin{align*}
\G([2],[2,5,8],[1,5,10]) & = \G([2],[2],[1]) \oplus  \big(\G([5,8]) \oplus \G([5,10])\big)\\
                         & = \G([2],[2],[1]) \oplus (8-5-1)   \oplus (10-5-1) \\
                         & = 3 \oplus 2 \oplus 4 = 5.
\end{align*}


\begin{corollary} \label{SSD.even}
If all strips have even numbers of tokens, the game {\sc Star Silver Dollar} is exactly the sum of its strips. In other
words, if $S^i = [a^i_1,a^i_2,\ldots,a^i_{2k_i}]$ for $1 \leq i \leq m$ then
\[\G(S^1, S^2, \ldots, S^m) = \G(S^1) \oplus \G(S^2) \oplus \cdots \oplus \G(S^m).\]
\end{corollary}

For example,
\begin{align*}
\G([2,5],[3,6,8,10]) & = \G([2,5]) \oplus \G([3,6,8,10]) \\
                     & = (5-2-1) \oplus \big((6-3-1) \oplus (10-8-1)\big) = 1.
\end{align*}

\begin{proof}[Proof of Corollary \ref{SSD.even}]
By Theorem \ref{SSD.MT}, it remains to prove that
\[\G(S^1_H, S^2_H, \ldots, S^m_H) = \G(S^1_H) \oplus \G(S^2_H) \oplus \cdots \oplus \G(S^m_H)\]
if every $S^i_H$ has exactly two tokens. The argument is exactly similar to the solution of {\sc Silver Dollar} with even
tokens.
\end{proof}


\section{$\P$-positions of {\sc Star Nim}}   \label{Ss.SN.P}

beginning simple cases, the two games are by far different, from winning strategy to Sprague-Grundy functions.

We characterize the $\P$-positions of $m$-{\sc Star Nim} for $m \leq 4$.
Except for the two positions $([0],[1])$ and $([1],[1])$, {\sc 2-Star Nim} and 2-pile {\sc Nim} have the same $\P$-positions.
if the payer who makes the last move loses).

\begin{proposition} \label{SN.P.2}
The {\sc 2-Star Nim} $([a_1],[a_2])$ is a $\P$-position if and only if either $a_1 + a_2 = 1$ or $a_1 = a_2 \geq 2$.
\end{proposition}

Note that the condition $a_1 = a_2$ in Proposition \ref{SN.P.2} is equivalent to $a_1 \oplus a_2 = 0$, which is similar to
$\P$-positions in {\sc Nim}.

The $\P$-positions of {\sc 3-Star Nim} coincide with those of 3-pile {\sc Nim}.

\begin{proposition} \label{SN.P.3}
The {\sc 3-Star Nim} position $([a_1],[a_2],[a_3])$ is a $\P$-position if and only if $a_1 \oplus a_2\oplus a_3 = 0$.
\end{proposition}
\begin{proof}
Set $\A = \{([a_1],[a_2],[a_3]) \mid a_1 \oplus a_2\oplus a_3 = 0\}$. We need to verify the following two facts:
\begin{enumerate} \itemsep0em
\item there is no move between two distinct positions in $\A$; and
\item from any position not in $\A$, there exists a move that terminates in $\A$.
\end{enumerate}
For (1), let $([a_1],[a_2],[a_3]) \in \A$. Then $a_1 \oplus a_2\oplus a_3 = 0$. For every $k < a_1$, moving from $[a_1]$ to
$[k]$ in the first strip results in position $([k],[a_2],[a_3]) \notin \A$  as $k \oplus a_2\oplus a_3 \neq 0$. Similarly,
moving in the second or third strip also results in positions not in $\A$.

For (2), let $([a_1],[a_2],[a_3]) \notin A$. Then $a_1 \oplus a_2\oplus a_3 \neq 0$. As shown by Bouton \cite{Bouton}, one
can find a $k$ such that
$k < a_1$ and $k \oplus a_2 \oplus a_3 = 0$, or
$k < a_2$ and $a_1 \oplus k \oplus a_3 = 0$, or
$k < a_3$ and $a_1 \oplus a_2\oplus k = 0$.
If the first (resp.~second, third) case holds, moving from $a_1$ (resp.~$a_2$, $a_3$) to $[k]$ in the first (resp.~second,
third) strip terminates in $\A$.
\end{proof}

It appears complicated to generalise the above results to 4-{\sc Star Nim}. We obtain a formula for $\P$-positions
$([a_1],[a_2],[a_3],[a_4])$ with only $a_1 = 1 \leq a_2 \leq a_3 \leq a_4$ as follows.

\begin{proposition} \label{SN.P.4}
The {\sc 4-Star Nim} position $([1],[a_2],[a_3],[a_4])$ with $1 \leq a_2 \leq a_3 \leq a_4$ is a $\P$-position if and only if
$a_2 \oplus a_3 \oplus a_4 = 0$.
\end{proposition}

\begin{proof}
Note that the {\sc 4-Star Nim} position $([0],[a],[b],[c])$ is exactly the 3-pile {\sc Nim} position $(a-1,b-1,c-1)$ and so
it is a $\P$-position if and only if $(a-1) \oplus (b-1) \oplus (c-1) = 0$. Set

$\B_0 = \{([0],[a_2],[a_3],[a_4]) \mid (a_2-1) \oplus (a_3-1) \oplus (a_4-1) = 0\}$,

$\B_1 = \{([1],[a_2],[a_3],[a_4]) \mid 1 \leq a_2 \leq a_3 \leq a_4, a_2 \oplus a_3 \oplus a_4 = 0\}$,

$\B = \B_0 \cup \B_1$.

We need to verify the following two facts:
\begin{enumerate} \itemsep0em
\item there is no move between two distinct positions in $\B$; and
\item from any position $([a_1],[a_2],[a_3],[a_4]) \notin \B$ with $a_1 \leq 1 \leq a_2 \leq a_3 \leq a_4$, there exists a
    move that terminates in $\B$.
\end{enumerate}

For (1), it can be verified that there is no move between positions in $\B_0$ and between positions in $\B_1$, using property
$x \oplus y \neq x' \oplus y$ if $x \neq x'$. We show that there is no move from a position in $\B_1$ to a position in
$\B_0$. Let $p = ([1],[a_2],[a_3],[a_4]) \in B_1$. Then $a_2 \oplus a_3 \oplus a_4 = 0$.

If we move from $[1]$ to $[0]$ in the first strip, the obtained position is $([0],[a_2],[a_3],[a_4])$. Note that $a_2 \oplus
a_3 \oplus a_4 = 0$ and $(a_2-1) \oplus (a_3-1) \oplus (a_4-1) = 0$ cannot hold at the same time. To see this, consider the
binary representations of the numbers $a_i$ and $a_{i-1}$.  If $a_1\oplus a_2\oplus a_3=0$, an even number of $a_2,a_3,a_4$
end in 1, so an odd number of $a_2-1,a_3-1,a_4-1$ end in 1, so $(a_1-1)\oplus(a_2-1)\oplus(a_3-1)\not=0$.

If we move from $[a_2]$ to $[0]$ in the second strip, the obtained position is $([0], [1], [a_3],[a_4])$ (after swapping the
first two strips). Assume by way of contradiction that this position belongs to $\B_0$. Then $(a_3-1) \oplus (a_4 - 1) = 0$
so that $a_3 = a_4$. This contradicts $a_2 \oplus a_3 \oplus a_4 = 0$ with $a_2 > 0$. If $a_2 > 1$ and we move from $[a_2]$
to $[1]$ in the second strip, the obtained position is $([1], [1], [a_3],[a_4])$.  Assume by way of contradiction that this
position belongs to $\B_1$. Then $1 \oplus a_3 \oplus a_4 = 0$, contradicting $a_2 \oplus a_3 \oplus a_4 = 0$ with $a_2 > 1$.

Similarly, one can check that a move from $[a_3]$ in the third strip or from $[a_4]$ in the fourth strip cannot lead $p$ to
either $\B_0$ or $\B_1$.

For (2), let $q = ([a_1],[a_2],[a_3],[a_4]) \notin \B$ with $a_1 \leq 1 \leq a_2 \leq a_3 \leq a_4$.
If $a_1 = 0$, $(a_2-1) \oplus (a_3-1) \oplus (a_4-1) \neq 0$ since $q \notin \B_0$. Consider the {\sc Nim} position $(a_2-1,
a_3-1, a_4-1)$. As shown by Bouton \cite{Bouton}, there exists $k$ such that
(1) $k < a_2-1$ and $k \oplus (a_3-1) \oplus (a_4-1) = 0$ or
(2) $k < a_3-1$ and $(a_2-1) \oplus k \oplus (a_4-1) = 0$ or
(3) $k < a_4-1$ and $(a_2-1) \oplus (a_3-1) \oplus k = 0$.
Without loss of generality, we can assume that the case (1) holds. Consider the move from $[a_2]$ to $[k+1]$ in the second
strip. One can check that the obtained position belongs to $\B_0$.

If $a_1 = 1$, $a_2 \oplus a_3 \oplus a_4 \neq 0$ since $q \notin \B_1$. Consider the {\sc Nim} position $(a_2, a_3, a_4)$. As
shown by Bouton \cite{Bouton}, there exists $k$ such that
(1) $k < a_2$ and $k \oplus a_3 \oplus a_4 = 0$ or
(2) $k < a_3$ and $a_2 \oplus k \oplus a_4 = 0$ or
(3) $k < a_4$ and $a_2 \oplus a_3 \oplus k = 0$.
Without loss of generality, we can assume that the case (1) holds. If $k = 0$, $a_3 = a_4 > 0$. The obtained position can be
represented as {\sc 4-Star Nim} position $([0],[1],[a_3],[a_4]) \in \B_0$. If $k > 0$, the obtained position is
$([1],[k],[a_3],[a_4]) \in \B_1$.
\end{proof}

Direct calculation shows that many $\P$-positions of {\sc Star Nim} also satisfy the condition for $\P$-positions in {\sc
Nim}. For example, there are 5089 $\P$-positions $([a_1],[a_2],[a_3],[a_4])$ with $a_1 \leq a_2 \leq a_3 \leq a_4 \leq 50$ in
{\sc Nim} of which 4593 are $\P$-positions in {\sc Star Nim}. There are 112 $\P$-positions $([a_1],[a_2],[a_3],[a_4],[a_5])$
with $a_1 \leq a_2 \leq a_3 \leq a_4 \leq a_5 \leq 10$ of which 79 are $\P$-positions in {\sc Star Nim}.


\section{Sprague-Grundy function of {\sc 2-Star Nim}}  \label{Ss.SN.propety}

In the rest of paper, we study {\sc 2-Star Nim}. For convenience, we omit square bracket $[.]$ in position format. A position
$([a],[b])$ will be denoted by $(a,b)$.

Table \ref{T1} provides the Sprague-Grundy values $\G(a,b)$ of {\sc $2$-Star Nim} for $0 \leq a \leq 10$, $1 \leq  b \leq
15$. We establish several properties of the extension of this table.
\begin{table}[ht]
\begin{center}
\begin{tabular}{c|ccccccccccccccc}
10    & &  & &  & & &  &  &  &0  &2   &3  &6  &21 &20\\
9     & &  & &  & & &  &  &0  &1  &5   &19 &20 &18 &21\\
8     & &  & &  & & &  &0  &2  &5  &4   &18 &17 &20 &19\\
7     & &  & &  & &  &0  &1  &3  &4  &14  &17 &18 &19 &11\\
6     & &  & &  &  &0  &2  &3  &4  &13 &12  &16 &10 &11 &18\\
5     & &  & &  &0  &1  &9  &10 &11 &8  &7   &15 &16 &17 &13\\
4     & &  & &0  &2  &9  &10 &11 &6  &7  &8   &14 &15 &16 &12\\
3     & &  &0 &1  &6  &8  &5  &9  &7  &12 &13  &10 &11 &15 &17\\
2     & &0  &4 &5  &3  &7  &8  &6  &10 &11 &9   &13 &14 &12 &16\\
1     &1 &2  &3 &4  &5  &6  &7  &8  &9  &10 &11  &12 &13 &14 &15\\
\hline
a/b   &1 &2  &3 &4  &5  &6  &7  &8  &9 &10  &11  &12 &13 &14 &15
 \end{tabular}
\caption{Sprague-Grundy values $\G(a,b)$s for $1 \leq a \leq b$, $a \leq 10, b \leq 15$}\label{T1}
\end{center}
\end{table}

We characterize 2-{\sc Star Nim} positions whose Sprague-Grundy values are at most 5.
Given some Sprague-Grundy value $g$, we give an algorithm that computes the sequence $\big((a_i,b_i)\big)_{i \geq 0}$ of {\sc
2-Star Nim} positions whose Sprague-Grundy values are $g$.
We show that the sequence $(b_i - a_i)_{i \geq 0}$ is ultimately periodic.
Finally, we prove the ultimately additive periodicity of rows (and similarly columns) of the matrix of Sprague-Grundy values
and conjecture the ultimate periodicity of diagonals parallel to the main diagonal. 

\subsection{{\sc $2$-Star Nim} positions of small Sprague-Grundy values}

We are able to characterize positions of {\sc $2$-Star Nim} whose Sprague-Grundy values are up to $5$ as follows.


\begin{proposition} \label{G0-5}
Denote by $G_n$ the set of positions whose Sprague-Grundy value is $n$. We have
\begin{enumerate} \itemsep0em
\item [\rm{(i)}] $G_0 = \{(0,1), (k,k) \mid k \geq 2\}$;
\item [\rm{(ii)}] $G_1 = \{(0,2), (1,1), (2k+1,2k+2) \mid k \geq 1\}$;
\item [\rm{(iii)}] $G_2 = \{(0,3), (1,2), (2k,2k+1) \mid k \geq 2\}$;
\item [\rm{(iv)}] $G_3 = \{(0,4), (1,3), (2,5), (4k+2+i,4k+2+i+2) \mid k \geq 1, 0 \leq i \leq 1\}$;
\item [\rm{(v)}] $G_4 = \{(0,5), (1,4), (2,3), (6,9), (7,10), (8,11), (4k+i,4k+i+2) \mid k \geq 3, 0 \leq i \leq 1\}$;
\item [\rm{(vi)}] $G_5 = \{(0,6), (1,5), (2,4), (3,7), (8,10), (9,11), (4k+i,4k+i+3) \mid k \geq 3, 0 \leq i \leq 2\}$;
\end{enumerate}
\end{proposition}
\begin{proof}
Recall that $\rm{(i)}$ is part of Proposition \ref{SN.P.2}. For each $n$ with $1 \leq n \leq 5$, one needs to verify two
facts: (1) there is no move between two positions in $G_n$, and (2) from any position not in $G_0 \cup G_1 \cup \ldots \cup
G_n$, there exists one move to some position in $G_n$. The verification of the claims is simple and we leave to the reader.
\end{proof}


\subsection{An algorithm for the Sprague-Grundy function for {\sc 2-Star Nim}}   \label{Ss.2SN.alg}
Denote by $S^g$ the sequence $\big((a^g_0,b^g_0), (a^g_1,b^g_1), (a^g_2,b^g_2), \ldots\big)$ with $a_i^g \leq b_i^g$, called
the {\it $g$-sequence}, of positions whose Sprague-Grundy values are $g$ such that $a^g_i < a^g_j$ if $i<j$. Note that the
set $G_n$ in Proposition \ref{G0-5} is exactly the set of elements in the sequence $S^n$. We establish an algorithm that
computes the sequence member $(a^g_n,b^g_n)$ from various $(a^k_i,b^k_i)$ for $k \leq g$. We assume $g \geq 1$ as the case
for $g = 0$ has been solved linearly  in Proposition \ref{SN.P.2}.

$k$-sequence $S^k$ for all $k < g$ and the first $n-1$ members $(a^g_l,b^g_l)$ for $l < n$. We introduce an algorithm that
computes $n^{\text{th}}$ sequence member $(a^g_n,b^g_n)$.

We denote the sequence computed by the algorithm by $T^g = \big((x^g_n,y^g_n)\big)_{n \geq 0}$, and will shortly show that
the sequences $S^g$ and $T^g$ are the same.


\begin{algorithm} \label{alg1} \

\smallskip
\noindent
Input: $g, n$. 

\noindent
Previously computed: $(x^g_i,y^g_i)$ for $i \leq n-1$; $(x^k_j,y^k_j)$ for $k \leq g$ and all $j$ such that $x^k_j \leq
x^g_n$.
\begin{enumerate}  \itemsep0em
\item $x^g_n \leftarrow \mex\{ x^g_i, y^g_i \mid 0 \leq i < n\}$.
\item $y^g_n \leftarrow x^g_n + z$, where $z$ is the smallest nonnegative integer such that
    \begin{enumerate}
    \item $x^g_n + z \notin \{y^g_l \mid 0 \leq l < n\}$, and
    \item $(x^g_n,x^g_n + z) \notin T^k$ for $k < g$.
    \end{enumerate}
\end{enumerate}

\smallskip
\noindent
Output: $(x^g_n,y^g_n)$. \qed
\end{algorithm}

\begin{remark} \label{WSG}
The algorithm has the same overall structure as the Wythoff-Sprague-Grundy (WSG) algorithm due to Blass and Fraenkel
\cite{Bla90}, for the game {\sc Wythoff}, although there are many differences in the details of analysis. \qed
\end{remark}


Later, Example \ref{Ex.alg1} shows the calculation of $(x^3_5, y^3_5)$ together with the data required, and is illustrated in
Figure \ref{F.alg1}. We now discuss Algorithm \ref{alg1}.

\begin{remark} \label{R.ij}\
\begin{enumerate}  \itemsep0em
\item By the definition of $\mex$, Step (1) of Algorithm \ref{alg1} implies that if $i < j$ then $x^g_i < x^g_j$.
\item 
    To compute the term $(x^g_n,y^g_n)$ of sequence $T^g$, we need the terms of each of the sequences $T^k$ (for $k \leq
    g$) up to a point where we can check if $(x^g_n,y^g_n) \notin T^k$. Note that $x^g_i < x^g_{n-1}$ for all $i < n-1$ and
    $x^g_n = \mex\{ x^g_i, y^g_i \mid 0 \leq i < n\}$ with at most $n$ values $y^g_i$. Therefore, $x^g_n \leq x^g_{n-1} + n
    + 1$. Thus, we assume esch $T^k$ has been constructed up to the position $(x^k_l, y^k_l)$  whose first entry $x^k_l$ is
    equal or greater than $x^g_{n-1} + n + 1$.
\item We have $n \leq x^g_n \leq 2n$. In fact, since $x^g_i < x^g_{i+1}$, we have $n \leq x^g_n$. Also in Step (1), the set
    $\{ x^g_i, y^g_i \mid 0 \leq i < n\}$ has at most $2n$ elements while the set $\{i \mid 0 \leq i \leq 2n\}$ has $2n+1$
    elements. Therefore, $\mex \{ x^g_i, y^g_i \mid 0 \leq i < n\} \leq 2n$. Thus, $n \leq x^g_n \leq 2n$.
\item We have $y^g_n \leq n + (x^g_n+1)g$. In fact, there are at most $n$ values $y^g_i$ that $y^g_n$ can not be assigned
    to. In each $T^k$, the worst case is when $T^k$ contains $x^g_n+1$ pairs, as mentioned in (2). The second entries of
    these pairs can not be assigned to $y^g_n$. For $0 \leq k \leq g-1$, there are such $(x^g_n+1)g$ values. Thus, $y^g_n
    \leq n + (x^g_n+1)g$. By (3), we also have $y^g_n \leq n + (2n+1)g$. \qed
\end{enumerate}
\end{remark}


\medskip

\begin{remark}
We analyse the complexity of Algorithm \ref{alg1}.
Step (1) first sorts the set $\{ x^g_i, y^g_i \mid 0 \leq i < n\}$ in increasing order, requiring \rm{O}$\big(n(\log n + \log
g)\big)$ steps using Radix Sort, since we have $n$ integers of \rm{O}$(\log n + \log g)$ bits each.
It then scans the list of up to $2n$ elements to find its $\mex$, requiring \rm{O}$(n)$ comparisons of these integers.
Thus, Step (1) takes \rm{O}$\big(n(\log n + \log g)\big)$ steps.
The construction of each $T^k$ by Algorithm 1 up to a point as mentioned in Remark \ref{R.ij} requires up to $x^g_n + 1$
pairs in the worst case. We have proved in Remark \ref{R.ij} that $n \leq x^g_n \leq 2n$. We compare $x^g_n$ with only the
first entries of these pairs, each requiring time \rm{O}$(\log n)$ until we find one entry equal to or greater than $x^g_n$.
This process requires \rm{O}$(n\log n)$ steps. If this condition is meet, we compare $y^g_n$ with the second entry, requiring
time \rm{O}$(\log n + \log g)$.
Thus, assuming we have previously computed the required terms of  $T^k$ for $0\leq k \leq g-1$, Step (2) costs \rm{O}$(ng\log
n)$ steps.
The total time complexity of Algorithm \ref{alg1} is therefore \rm{O}$\big(n(\log n + \log g + g \log n)\big)$. This is
exponential in the sizes of the inputs $n$ and $g$. \qed
\end{remark}

Algorithm \ref{alg1} could be used as part of an iterative scheme to compute the $(x^g_n, y^g_n)$ in turn, provided the
computation is organised so that the required previously computed pairs are always available.

\begin{theorem}
The sequence $T^g$ computed by Algorithm \ref{alg1} coincides with the $g$-sequence $S^g$.
\end{theorem}
\begin{proof}
We first verify the termination of the algorithm. Step (1) requires at most $(2n)^2$ comparisons. For each $z$ in Step (2),
$x^g_n + z$ needs to be compared with no more than $n$ values of $y^g_i$ and $(x^g_n,y^g_n)$ needs to be compared with no
more than $g(x^g_n+1)$ positions $(x^k_m,y^k_m)$ of $T^k$ with $k < g$ and $m \leq x^g_n$. We can ignore all other
$(x^k_m,y^k_m)$ with $m > x^g_n$ since $x^k_m \geq m > x^g_n$ by Remark \ref{R.ij}. Now $z \leq 1 + \max\{y^g_i, y^k_m \mid i
< n, k < g, m \leq x^g_n\}$. Therefore, the algorithm must terminate.

We prove by induction on $g$ that the sequence $T^g$ computed by Algorithm \ref{alg1} coincides with the $g$-sequence $S^g$.
We need to verify the following conditions:

\rm{(i)} $T^g \cap (\cup_{k < g} T^k) = \emptyset$;

\rm{(ii)} there is no move between two positions in $T^g$;

\rm{(iii)} from every position not in $\cup_{k \leq g} T^k$, there exists one move that terminates in $T^g$.

The condition \rm{(i)} holds since in the step (2) of the algorithm, $(x^g_n,y^g_n)$ is chosen such that it has not appeared
in any sequence $T^k$ for $k < g$.

For \rm{(ii)}, note that if $(x,y)$ and $(x',y')$ belong to $T^g$ then $x \neq x'$ by Step (1) and $y \neq y'$ by Step (2).
Thus, \rm{(ii)} holds.

For \rm{(iii)}, let $(x,y) \notin \cup_{l \leq g} T^l$ with $x \leq y$. We show that there is a move from $(x,y)$ to some
position in $T^g$. Step (1) implies that the set $\{x^g_i, y^g_i \mid i \geq 0\}$ contains all nonnegative integers. It
follows that there exists some index $i_0$ such that either $x = x^g_{i_0}$ or $x = y^g_{i_0}$.

If the latter case holds, we have $y > x^g_{i_0}$ since $y \geq x = y^g_{i_0} \geq x^g_{i_0}$ and $(x,y) \neq
(x^g_{i_0},y^g_{i_0})$. One can move the token in the square $y$ to the square $x^g_{i_0}$, terminating in $T^g$.

If the former case holds, we compare $y$ with $y^g_{i_0}$. Note that $y \neq y^g_{i_0}$. If $y > y^g_{i_0}$, one can move the
token in the square $y$ to the square $y^g_{i_0}$, terminating in $T^g$. We now consider the case $y < y^g_{i_0}$.

We claim that $y \in \{y^g_i \mid i < i_0\}$. Since $(x_{i_0}^g,y) = (x,y) \notin \cup_{l \leq g} T^l$, if $y \notin \{y^g_i
\mid i < i_0\}$ then Step (2) must choose $y^g_{i_0} \leq y$, giving a contradiction.

Thus, there exists $j_0 < i_0$ such that $(x^g_{j_0}, y^g_{j_0}) \in T^g$ and $y^g_{j_0} = y$. Also note that $j_0 < i_0$
implies $x^g_{j_0} < x^g_{i_0}$ by Remark \ref{R.ij}. Thus, moving the token from the square $x = x^g_{i_0}$ to the square
$x^g_{j_0}$ results in the position $(x^g_{j_0}, y^g_{j_0}) \in T^g$. 
\end{proof}

\begin{example} \label{Ex.alg1}
We illustrate data required to run Algorithm \ref{alg1} in Figure \ref{F.alg1}. To calculate $(x^3_5,y^3_5)$, step (1) gives
$x^3_5 = 10$. In step (2), we start with $z = 0$, giving $x^3_5 + z = 10$ satisfying (2a). But $(10,10)$ appears in $T^0$ so
step (2b) fails. By increasing $z$ to $1$ we have $(10,11)$ appearing in $T^2$, so step (2b) fails. Then we increase $z$ to
$2$, giving $(10,12)$ and satisfying step (2). Therefore, $y^3_5 = 12$. \qed
\end{example}

\bigskip

\begin{figure}[ht]
\begin{center}
\begin{tikzpicture}

\node at (0,-1) {$T^0$};
\node at (0,0) {(0,1)};
\node at (0,1) {(2,2)};
\node at (0,2) {(3,3)};
\node at (0,3) {(4,4)};
\node at (0,4) {(5,5)};
\node at (0,5) {(6,6)};
\node at (0,6) {(7,7)};
\node at (0,7) {(8,8)};
\node at (0,8) {(9,9)};
\node at (0,9) {(10,10)};

\node at (2,-1) {$T^1$};
\node at (2,0) {(0,2)};
\node at (2,1) {(1,1)};
\node at (2,2) {(3,4)};
\node at (2,3) {(5,6)};
\node at (2,4) {(7,8)};
\node at (2,5) {(9,10)};
\node at (2,6) {(11,12)};

\node at (4,-1) {$T^2$};
\node at (4,0) {(0,3)};
\node at (4,1) {(1,2)};
\node at (4,2) {(4,5)};
\node at (4,3) {(6,7)};
\node at (4,4) {(8,9)};
\node at (4,5) {(10,11)};

\node at (6,-1) {$T^3$};
\node at (6,0) {(0,4)};
\node at (6,1) {(1,3)};
\node at (6,2) {(2,5)};
\node at (6,3) {(6,8)};
\node at (6,4) {(7,9)};
\draw [black, fill=cyan] (5.3,5.5) -- (6.7,5.5) -- (6.7,4.7) -- (5.3,4.7) -- (5.3,5.5);
\node at (6,5) {(10,12)};

\end{tikzpicture}
\caption{Finding $(x^3_5,y^3_5) = (10,12)$ using Algorithm \ref{alg1}, including sufficient data required.} \label{F.alg1}
\end{center}
\end{figure}


We now improve Algorithm \ref{alg1}. Note that to compute $(x^g_n, y^g_n)$ in the sequence $T^g$, Algorithm \ref{alg1}
computes a set whose size is $\Theta(n)$. 
We will show that there are several ways we can improve Algorithm \ref{alg1}. In step (1), we can replace the existing set by
a set of size at most $g+2$. In step (2a), we can replace the existing set by a set of size at most $g+1$. In step (2b),
instead of scanning all pairs in $T^k$, we only need to scan one pair. Thus, the algorithm needs to scan only \rm{O}$(g^2)$
values.

\begin{lemma} \label{L12}
$\G(0,b) = b-1$, $\G(1,b) = b$, $\G(2,2) = 0$, and for $b \geq 3$
\[\G(2,b) =
\begin{cases}
b+1, \quad \hbox{ if $b \equiv 0, 1 \pmod{3}$}; \\
b-2, \quad \hbox{ otherwise}.
\end{cases}
\]
\end{lemma}
\begin{proof}
The lemma can be proved by induction on $b$.
\end{proof}


We first prove lower and upper bounds for the Sprague-Grundy function.

\begin{proposition} \label{SN.2.bound}
For {\sc $2$-Star Nim}, we have $b-a \leq \G(a,b) \leq b+a-1$ for positive integers $a, b$.
\end{proposition}
\begin{proof}
We first show that $\G(a,b) \leq b+a-1$. The claim can be proved by induction on $a$. Note that the claim holds for $a \leq
2$ by Lemma \ref{L12}. The remaining argument is straightforward and we leave it to the reader.

holds for all $a \leq n$ for some $n \geq 2$. We show that the theorem holds for $a = n+1$. Clearly the theorem holds for $b
\leq a + 2$. Assume that the theorem holds for $b \leq m$ for some $m \geq a+2$. We show that the theorem holds for $a=n+1$,
$b = m+1$. We have
that $\G(a,b) \leq a+b-2$, as required. By the principle of induction, $\G(a,b) \leq b+a-1$ for all $a, b$.

We next prove that $\G(a,b) \geq b-a$. Assume by way of contradiction that $\G(a,b) < b-a$ for some $a$ and $b$ with $a \leq
b$. We can choose the smallest such integer $a$. Moreover, corresponding to this $a$, we can choose the smallest such integer
$b$. Assume that $\G(a,b) = v$ for some $v < b-a$. Consider the position $(a,a+v)$. Since $b$ is the smallest integer such
that $\G(a,b) < b-a$ and $a+v < b$, we have $\G(a,a+v) \geq (a+v)-a = v$. Since one can move from $(a,b)$ to $(a,a+v)$,
$\G(a,a+v)\neq v$ and so $\G(a,a+v) > v$. Therefore, there exists some position $(c,d)$ with $c \leq d$ such that one can
move from $(a,a+v)$ to $(c,d)$ and $\G(c,d) = v$. We have either $c = a$ and $d < a+v$ or $c < a$ and $d = a+v$.

The former case cannot hold as $\G(a,b) = v = \G(c,d)$ and there exists a move from $(a,b)$ to $(c,d)$ (moving the token on
the square $b$ to the square $d$). Therefore, the latter case holds. By the inductive hypothesis on $c < a$, we have $\G(c,e)
\geq e-c$ for all $e$ and so $\G(c,d) \geq d-c = a + v - c > v$, giving a contradiction.
\end{proof}

\medskip


We introduce two concepts that we will use to improve Algorithm \ref{alg1} by reducing the amount of data used. In Step 1 of
Algorithm \ref{alg1}, we can skip all values smaller than $x^g_{n-1}$, then count from $x^g_{n-1}$ to find the smallest
integer not in the set. Similarly, in step (2a), we can start comparing with those values not smaller than $x^g_{n-1}$ rather
than all $y^g_l$. This results in the following concept.

\begin{definition} \label{D.mex.min}
Let $S$ be a finite set of nonnegative integers and let $a$ be a nonnegative integer. Set $\mex_{a\Uparrow}(S) = \min\{b\mid
b \geq a, b \notin S\}$.    \qed
\end{definition}
For example, if $S = \{3,4,5,7,9\}$ then $\mex_{2\Uparrow}(S) = 2$, $\mex_{3\Uparrow}(S) = 6$.

\medskip


In step (2b), there is at most one value $x^k_l$ equal to $x^g_n$. We know the first entries in each $T^k$ are increasing.
Suppose we have already scanned the first $m$ pairs in $T^k$ whose first entries are smaller than $x^g_{n-1}$ to compute
$y^g_{n-1}$. Then when computing $x^g_n$, we can skip these scanned values. We would like to mark a position to record that
we have already scanned up to there. The following concept helps us to do so.

\begin{definition} \label{D.f.v}
For each $v$, there exists some $x^k_f$ in column $T^k$ such that $x^k_f \leq v < x^k_{f+1}$. For convenience, denote by
$f_k(v)$ this index $f$. The existence of $f_k(v)$ is straightforward since $x^k_i < x^k_{i+1}$ by Remark \ref{R.ij}. \qed
\end{definition}

For example, one can check some first elements of $T^k$ for $k \leq 2$ as follows:
\begin{align*}
T^0 & = ((0,1),(2,2),(3,3),(4,4),(5,5),(6,6),(7,7), \ldots), \\
T^1 & = ((0,2),(1,1),(3,4),(5,6),(7,8),(9,10),(11,12), \ldots),\\
T^2 & = ((0,3),(1,2),(4,5),(6,7),(8,9),(10,11),(12,13), \ldots).
\end{align*}
Let $v = 6$. Then $f_0(6) = 5$ (since $x^0_5 = 6$, $x^0_6 = 7$), $f_1(6) = 3$ (since $x^1_3 = 5$, $x^1_4 = 7$), and $f_2(6) =
3$ (since $x^2_3 = 6$, $x^2_4 = 8$).

\medskip

We now show that the sets scanned in Algorithm \ref{alg1} can be replaced by new sets of smaller sizes.

\begin{lemma} \label{L.alg1}
The following properties hold for Algorithm \ref{alg1}:
\begin{enumerate} \itemsep0em
\item [\rm{(1)}] $x^g_n  = \mex_{{x^g_{n-1}}\Uparrow} \{ x^g_{n-1}, y^g_{n-1-i}\mid 0 \leq i \leq g,  \ \  y^g_{n-1-i} \geq
    x^g_{n-1}\}$;
\item [\rm{(2)}] $z$ is the smallest nonnegative integer such that
    \begin{enumerate}
    \item [\rm{(a)}] $x^g_n + z \notin \{y^g_{n-1-i}\mid 0 \leq i \leq g, \ \  y^g_{n-1-i} \geq x^g_{n-1}\}$, and
    \item [\rm{(b)}] $(x^g_n,x^g_n + z) \neq (x^k_{f_k(x^g_n)},y^k_{f_k(x^g_n)})$.
    \end{enumerate}
\end{enumerate}
\end{lemma}
\begin{proof}
\item \rm{(1)} By Remark \ref{R.ij}, $x^g_i < x^g_n$ if $i < n$ and moreover $a \in \{ x^g_i, y^g_i \mid i < n\}$ if $a \leq
    x^g_{n-1}$. Therefore,
    $$x^g_n = \mex\{ x^g_i, y^g_i \mid i < n\} = \mex_{x^g_{n-1}\Uparrow}\{ x^g_i, y^g_i \mid i < n\}.$$
    In the set $\{ x^g_i, y^g_i \mid i < n\}$ on the right-hand side of the last equation, we can exclude all members smaller
    than $x^g_{n-1}$. We now do this task.
    \begin{enumerate} \itemsep0em
    \item [\rm{(i)}]  By Remark \ref{R.ij}, $x^g_i < x^g_{n-1}$ for $i < n-1$ and so we can exclude these $x^g_i$.
    \item [\rm{(ii)}] For $i > g$, by Remark \ref{R.ij}, $x^g_{n-1-i} \leq x^g_{n-1}-i < x^g_{n-1} - g$. Combining this and
        Proposition \ref{SN.2.bound} gives $y^g_{n-1-i} \leq x^g_{n-1-i} + g < x^g_{n-1}$. Thus, we can exclude all
        $y^g_{n-1-i}$ with $i > g$.
        Finally, we exclude those $y^g_{n-1-i} < x^g_{n-1}$ for $i \leq g$.
    \end{enumerate}
     Therefore, (1) holds.
\item \rm{(2)} By Remark \ref{R.ij}, $x^g_{n-1} < x^g_n \leq x^g_n + z$. For $i > g$, it has been shown in \rm{(ii)} above
    that $y^g_{n-1-i} < x^g_{n-1}$ and so $y^g_{n-1-i} < x^g_n + z$. Therefore, in Step (2a) of Algorithm \ref{alg1}, $x^g_n
    + z \notin \{y^g_l\mid l \leq n\}$ if and only if $x^g_n + z \notin \{y^g_{n-1-i}\mid 0 \leq i \leq g\}$.

\medskip

    For the condition \rm{(2b)} in Algorithm \ref{alg2}, we need to compare $(x^g_n,x^g_n + z)$ with only
    $(x^k_{f_k(x^g_n)},y^k_{f_k(x^g_n)})$. This is because $x^k_{f_k(x^g_n) - i} < x^k_{f_k(x^g_n)} \leq x^g_n <
    x^k_{f_k(x^g_n)+j}$ for all $i, j \geq 1$.
\end{proof}

\medskip

We have following bounds for $x^g_n$.

\begin{corollary} \label{SN.2.bound.alg}
For given $n, g$, we have $x^g_n \leq x^g_{n-1} + g + 2$. 
\end{corollary}
\begin{proof}
Therefore, $n \leq x^g_n$. Also in step 1, the set $\{ x^g_i, y^g_i \mid i < n\}$ has at most $2n$ elements while the set
$\{i \mid 0 \leq i \leq 2n\}$ has $2n+1$ elements. Therefore, $\mex \{ x^g_i, y^g_i \mid i < n\} \leq 2n$.
In Lemma \ref{L.alg1}, the set $\{ x^g_{n-1}, y^g_{n-1-i}\mid 0 \leq i \leq g,  \ \  y^g_{n-1-i} \geq x^g_{n-1}\}$ has at
most $g+1$ values different to $x^g_{n-1}$. Therefore, $x^g_n \leq x^g_{n-1} + g + 2$.
reaches the minimum $n$, giving the gap $x^g_n - x^g_{n-1} = n+1$.
\end{proof}

We now improve Algorithm \ref{alg1} so that we need only \rm{O}$(g^2)$ values for input.

\begin{definition} \label{Def.T.Y}
For all $k < g$. Set\\
$\ \ \ \ \ \ \ \ T^g_{n-1} = \{x^g_{n-1},y^g_{n-1}\}$,\\
$\quad \quad \quad \quad Y^g_{n-1} = \{y^g_{n-1-i} \mid 0 \leq i \leq g, \ \  y^g_{n-1-i} \geq x^g_{n-1}\}$,\\
$\quad \quad \quad \quad T^k_{n-1} = \{x^k_{f_k(x^g_{n-1})},y^k_{f_k(x^g_{n-1})}\}$,\\
$\quad \quad \quad \quad Y^k_{n-1} = \{y^k_{f_k(x^g_{n-1})-i} \mid 0 \leq i \leq k, \ \  y^k_{f_k(x^g_{n-1})-i} \geq
x^k_{f_k(x^g_{n-1})}\}$. \qed
\end{definition}

\begin{algorithm} \label{alg2} \

\smallskip
\noindent
Input: $g, n$.

\noindent
Previously computed: $(x^g_{n-1},y^g_{n-1}), T^i_{n-1}, Y^i_{n-1} \ \ \ (0 \leq i \leq g)$.
\begin{enumerate}  \itemsep0em
\item $x^g_{n} \leftarrow \mex_{x^g_{n-1}\Uparrow}(T^g_{n-1} \cup Y^g_{n-1})$.
\item $y^g_{n} \leftarrow x^g_{n} + z$, where $z$ is the smallest nonnegative integer such that
    \begin{enumerate}
    \item $x^g_{n} + z \notin Y^g_{n-1}$, and
    \item $(x^g_{n},x^g_{n} + z) \neq \big(\min(T^k_{n}), \max(T^k_{n})\big)$ for $k < g$.
    \end{enumerate}
\item $T^g_n = \{x^g_n, y^g_n\}, Y^g_n = (Y^g_{n-1} \setminus\{z \in Y^g_{n-1} \mid z < x^g_n\}) \cup \{y^g_{n}\}$.
\end{enumerate}

\smallskip
\noindent
Output: $(x^g_n, y^g_n), T^i_n, Y^i_n$ with $0 \leq i \leq g$. \qed
\end{algorithm}

\medskip

The next result follows from Lemma \ref{L.alg1}.

\begin{lemma}  \label{L.alg2}
Algorithm \ref{alg2} computes $x^g_{n}, y^g_{n}$, and $T^i_n, Y^i_n$ with $0 \leq i \leq g$.
\end{lemma}

\begin{example}
We illustrate the data required to run Algorithm \ref{alg2} in Figure \ref{F.alg2}. To calculate $(x^3_5, y^3_5)$, the input,
illustrated in given white box, includes
$(x^3_4, y^3_4) = (7,9)$,
$T^0_4 = \{7\}$,    $Y^0_4 = \{7\}$,
$T^1_4 = \{7, 8\}$, $Y^1_4 = \{8\}$,
$T^2_4 = \{6, 7\}$, $Y^2_4 = \{7\}$,
$T^3_4 = \{7, 9\}$, $Y^0_4 = \{8, 9\}$.

Step (1) gives $x^3_5 = 10$ as the smallest integer bigger than $x^3_4 = 7$ and not in the set $T^3_4 \cup Y^3_4 = \{7, 8,
9\}$. 

In step (2), we recursively calculate $T^k_5, Y^k_5$ for $k \leq 2$. These new sets are in colored boxes in the first three
columns.
$T^0_5 = \{10\}$,     $Y^0_5 = \{10\}$,
$T^1_5 = \{9, 10\}$,  $Y^1_5 = \{10\}$,
$T^2_5 = \{10, 11\}$, $Y^2_5 = \{11\}$.
We can now ignore $T^k_4, Y^k_4$ for $k \leq 2$. Let us take the opportunity to discuss the role of $f_k(v)$ in Definition
\ref{D.f.v}. Let $v = 10$, based on Definition \ref{D.f.v}, $f_0(10)$ is the index such that to $x^0_{f_0(10)} \leq 10$ and
$x^0_{f_0(10)+1} > 10$. This gives $x^0_{f_0(10)} = 10$ or $f_0(10) = 9$. Also in Figure \ref{F.alg2}, one can see that we
have replaced the pair $(7, 7)$ by $(8, 8)$ and then $(9, 9)$ and end with $(10, 10)$ before assigning $\{10\}$ to  $T^0_5$
and $\{10\}$ to  $Y^0_5$. The same process is also applied for $T^1_5, Y^1_5, T^2_5, Y^2_5$.

In step (2), we next start with $z = 0$, giving $x^3_5 + z = 10$. But $(10,10)$ appears in $T^0_5$. By increasing $z$ to $1$
we have $(10,11)$ appearing in $T^2_5$. Then we increase $z$ to $2$, giving $(10,12)$ and satisfying step (2). Therefore,
$y^3_5 = 12$.

In step (3), we assign $T^3_5 = \{10, 12\}$ and $Y^3_5 = \{12\}$. \qed
\end{example}


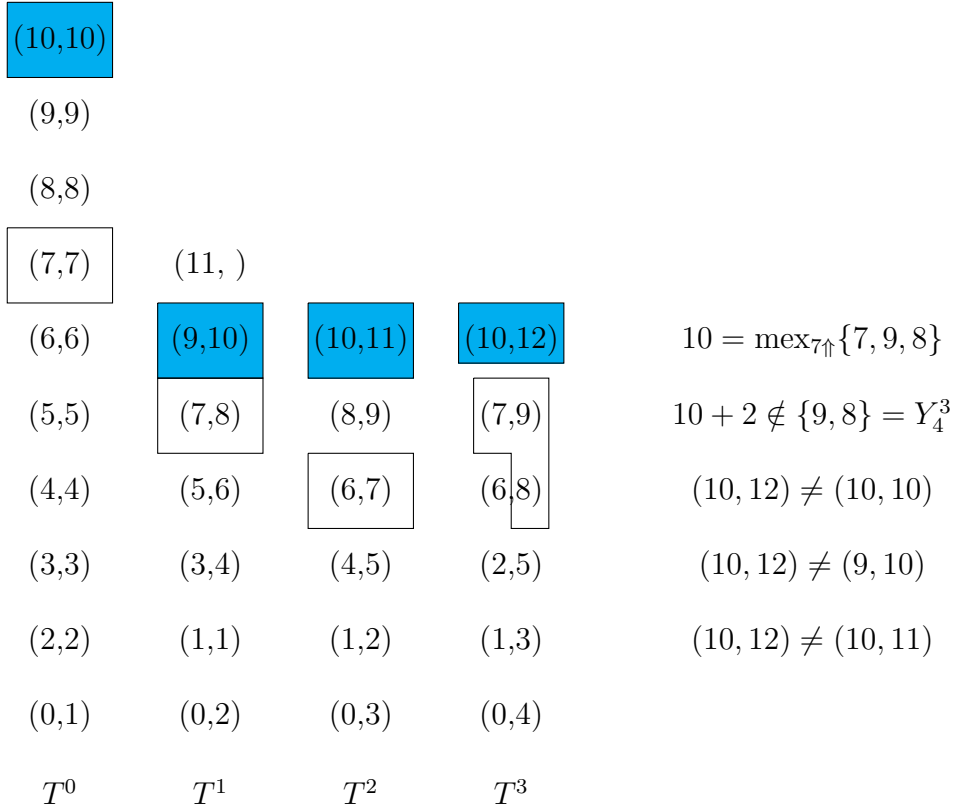
\begin{figure}[ht]
\begin{center}
\begin{tikzpicture}

\node at (0,-1) {$T^0$};
\node at (0,0) {(0,1)};
\node at (0,1) {(2,2)};
\node at (0,2) {(3,3)};
\node at (0,3) {(4,4)};
\node at (0,4) {(5,5)};
\node at (0,5) {(6,6)};
\node at (0,6) {(7,7)};
\node at (0,7) {(8,8)};
\node at (0,8) {(9,9)};
\node at (0,9) {(10,10)};
\draw [black, fill=cyan] (-.7,9.5) -- (0.7,9.5) -- (0.7,8.5) -- (-.7,8.5) -- (-.7,9.5);
\node at (0,9) {(10,10)};
\draw [black] (-.7,6.5) -- (0.7,6.5) -- (0.7,5.5) -- (-.7,5.5) -- (-.7,6.5);

\node at (2,-1) {$T^1$};
\node at (2,0) {(0,2)};
\node at (2,1) {(1,1)};
\node at (2,2) {(3,4)};
\node at (2,3) {(5,6)};
\node at (2,4) {(7,8)};
\node at (2,5) {(9,10)};
\node at (2,6) {(11, )};
\draw [black, fill=cyan] (1.3,5.5) -- (2.7,5.5) -- (2.7,4.5) -- (1.3,4.5)--(1.3,5.5);
\node at (2,5) {(9,10)};
\draw [black] (1.3,4.5) -- (2.7,4.5) -- (2.7,3.5) -- (1.3,3.5)--(1.3,4.5);

\node at (4,-1) {$T^2$};
\node at (4,0) {(0,3)};
\node at (4,1) {(1,2)};
\node at (4,2) {(4,5)};
\node at (4,3) {(6,7)};
\node at (4,4) {(8,9)};
\draw [black] (3.3,3.5) -- (4.7,3.5) -- (4.7,2.5) -- (3.3,2.5)--(3.3,3.5);
\draw [black, fill=cyan] (3.3,5.5) -- (4.7,5.5) -- (4.7,4.5) -- (3.3,4.5)--(3.3,5.5);
\node at (4,5) {(10,11)};

\node at (6,-1) {$T^3$};
\node at (6,0) {(0,4)};
\node at (6,1) {(1,3)};
\node at (6,2) {(2,5)};
\node at (6,3) {(6,8)};
\node at (6,4) {(7,9)};
\draw [black, fill=cyan] (5.3,5.5) -- (6.7,5.5) -- (6.7,4.7) -- (5.3,4.7) -- (5.3,5.5);
\node at (6,5) {(10,12)};
\draw [black] (5.5,4.5) -- (6.5,4.5) -- (6.5,2.5) -- (6,2.5) -- (6.0,3.5)--(5.5,3.5)-- (5.5,4.5);

\node at (10,5) {$10 = \mex_{{7}\Uparrow} \{ 7,9,8\}$};
\node at (10,4) {$10 + 2 \notin \{9,8\} = Y^3_4$};
\node at (10,3) {$(10,12) \neq (10,10)$};
\node at (10,2) {$(10,12) \neq (9,10)$};
\node at (10,1) {$(10,12) \neq (10,11)$};

\end{tikzpicture}
\caption{Finding $(x^3_5,y^3_5) = (10,12)$ using Algorithm \ref{alg2}, using only the data in boxes.
(We do not need to calculate the second entry on the top of column 2.)} \label{F.alg2}
\end{center}
\end{figure}


\begin{remark} \label{R.alg2.linear}
There are three improvements in Algorithm \ref{alg2}.
\begin{enumerate} \itemsep0em
\item the set $T^g_{n-1} \cup Y^g_{n-1}$ has at most $g+2$ elements;
\item $x^g_{n} + z$ has to be compared with only $g+1$ values ($y^g_{n-i}$ for $i \leq g+1$) rather than $n$ values
    $y^g_{n-i}$ as in Algorithm \ref{alg1};
\item $(x^g_n, x^g_n+z)$ has to be compared with only $g$ pairs $(x^k_{f_k(x^g_{n})},y^k_{f_k(x^g_{n})})$ for $0 \leq k
    \leq g-1$ rather than up to $(x^g_n+1)g$ pairs $(x^k_m,y^k_m)$ with $k < g$ as in Algorithm \ref{alg1} (Remark
    \ref{R.ij}).  \qed 
    $g(x^g_{n}+1)$ pairs to be compared.
\end{enumerate}
\end{remark}


\begin{remark} \label{R.alg2}
Note that $|T^k_n| \leq 2$, $|Y^k_n| \leq k+1 \leq g+1$ and $T^k_n \cap Y^k_n \neq \emptyset$, and so Algorithm \ref{alg2}
stores only $\sum_{k = 0}^{g} (k+2) = (g^2 + 5g + 4)/2$ values in $T^g_{n-1} \cup Y^g_{n-1}$ and $\bigcup_{k = 0}^{g-1}
\big(T^k_n \cup Y^k_n\big)$. By Remark \ref{R.ij} and Corollary \ref{SN.2.bound.alg}, these values are all at most $2n+g+1$.
We will use this finiteness to prove a periodicity property in the next section. \qed
\end{remark}



\subsection{A periodicity property of $g$-sequences}
We discuss the pattern of $b^g_n - a^g_n$ in the $g$-sequence $S^g$ when $n$ increases. One can see in Proposition \ref{G0-5}
that for $g \leq 5$, $b^g_n - a^g_n$ will eventually be a constant. When $g$ is large enough, this need not happen.
Nonetheless, we will show that $(b^g_n - a^g_n)_{n \geq 0}$ is eventually periodic. For example, for $g = 6$, the sequence
$(b^6_n - a^6_n)_{n \geq 5}$ is the repetition of ``3, 3, 4, 3, 3, 4, 3, 3, 4, 4, 4, 4".

Recall that a sequence $(s_n)_{n \geq 0}$ is said to be {\it ultimately periodic} if there exist $n_0$ and $p$ such that
$s_{n+p} = s_n$ for all $n \geq n_0$.

We will modify Algorithm \ref{alg2} to obtain a new algorithm computing $y^g_n-x^g_n$. In the new algorithm, we will overcome
the increasing values in $T^k_n, Y^k_n$ so that we only need to store bounded values, making the algorithm able to be
implemented on a finite state machine. This helps us to prove the periodicity of the difference $y^g_n-x^g_n$. The following
lemma gives bounds on the values used in Algorithm 2.


\medskip

\begin{lemma} \label{L.alg3.bound}
Algorithm \ref{alg2} works with numbers at most $x^g_n + g + 1$ and at least $x^g_n - 2g - 2$.
\end{lemma}
\begin{proof}
In column $T^g$, the first entries $x^g_i$ are increasing. For all $i \leq n$, by Proposition \ref{SN.2.bound}, we have
$y^g_i \leq x^g_i +g \leq x^g_n + g$. 
In each column $T^k$, we need to calculate up to the point $x^k_{f_k(x^g_n)+1}$ (the first one from the bottom such that the
first entry is greater than $x^g_n$) for the first entries and $y^k_{f_k(x^g_n)}$ for the second entries. By Corollary
\ref{SN.2.bound.alg}, we have $x^k_{f_k(x^g_n)+1} \leq x^k_{f_k(x^g_n)}+k+2 \leq x^g_n + g+1$. For the second entries, by
Proposition \ref{SN.2.bound}, we have $y^k_i \leq x^k_i +k \leq x^k_{f_k(x^g_n)} + k \leq x^g_n + g-1$. 
\ref{SN.2.bound.alg}.

The smallest value used by the algorithm is
$\min\{x^k_{f_k(x^g_{n-1})} \mid 0 \leq k \leq g\}$.
Recall that
$x^k_{f_k(x^g_{n-1})} \geq x^k_{f_k(x^g_{n-1})+1} - k - 2$ by Corollary \ref{SN.2.bound.alg}.
By Definition \ref{D.f.v}, we have $x^k_{f_k(x^g_{n-1})+1} > x^g_{n-1}$.
Also by Corollary \ref{SN.2.bound.alg}, $x^g_{n-1} \geq x^g_{n}-g-2$. Therefore,
$x^k_{f_k(x^g_{n-1})} \geq x^g_n - g - k -3 \geq x^g_n - 2g - 2$.
\end{proof}

\begin{definition} \label{D.alg3}
{\it Algorithm} 3 is obtained from Algorithm \ref{alg2} as follows. Let $b_{n -1}= \min_{i = 0}^{g}(T^i_{n-1})$, $x'^g_{n-1}
= x^g_{n-1} - b_{n-1}$, $y'^g_{n-1} = y^g_{n-1} - b_{n-1}$, $T'^i_{n-1} = \{x - b_{n-1} \mid x \in T^i_{n-1}\}$, $Y'^i_{n-1}
= \{x - b_{n-1} \mid x \in Y^i_{n-1}\}$. We replace $(x^g_{n-1},y^g_{n-1}), T^i_{n-1}, Y^i_{n-1} \ \ (0 \leq i \leq g)$ by
$(x'^g_{n-1},y'^g_{n-1}), T'^i_{n-1}, Y'^i_{n-1} \ \ (0 \leq i \leq g)$ respectively.
\qed \end{definition}

\medskip
The following lemma is straightforward.

\begin{lemma} \label{L.alg30}
The output of Algorithm 3 is obtained by subtracting $b_{n-1}$ from every value in the output of Algorithm \ref{alg2}.
\end{lemma}


We  prove the following.

\begin{theorem} \label{g-sequence-period}
For every $g$-sequence $S^g = (a^g_n,b^g_n)_{n \geq 0}$, the sequence $(b^g_n - a^g_n)_{n \geq 0}$ is ultimately periodic.
\end{theorem}

\begin{proof}
By Remark \ref{R.alg2}, Algorithm \ref{alg2} has used at most $(g^2+5g+4)/2$ stored values. By Lemma \ref{L.alg3.bound},
values used in Algorithm \ref{alg2} are between $x^g_n - 2g - 2$ and $x^g_n + g + 1$. Therefore, values used in Algorithm 3
are between $0$ and $3g+4$. Here $3g+4$ is the number of values in the integer interval $[x^g_n - 2g - 2, x^g_n + g + 1]$.
Hence, Algorithm 3 can be implemented on a finite state machine. The algorithm repeatedly returns values $b^g_n - a^g_n$ and
so the sequence $(b^g_n - a^g_n)_{n \geq 0}$ is ultimately periodic.
\end{proof}



\subsection{The ultimately additive periodicity of Sprague-Grundy values}

A sequence $(s_n)_{n \geq 0}$ is said to be {\it additively periodic} if there exist $n_0$ and $p$ such that $s_{n+p} =
s_n+p$ for all $n \geq n_0$ \cite{les}.

We show that in the expanded table of Table \ref{T1}, every row (column) is ultimately additively periodic. For example, for
$a = 1, 2, 3, 4, 5, 6$, the sequence $(\G(a,b))_{b\geq 0}$ is ultimately additively periodic with the period $p = 1, 3, 9,
36, 144, 720$,
respectively. There is a remarkable pattern with these periods $p$. The pattern is as follows:
\begin{align*}
\begin{cases}
p_2 = 3 p_1, \\
p_3 = 3 p_2, \\
p_4 = 4 p_3, \\
p_5 = 4 p_4, \\
p_6 = 5 p_5. \\
\end{cases}
\end{align*}
We do not yet know $p_7$, so we do not know if this pattern continues.

\medskip

Ultimately additive periodicity has been found for the {\sc Wythoff} game and some of its variants \cite{Dre99, Ho12,
Jiao.13, Landman}. In these variants, the players alternately move from a position $(a, b)$, following some given rules. For
example, a move in {\sc Wythoff} is one of following options: $(a, b) \rightarrow (a-i,b)$ with $1 \leq i < a$; $(a, b)
\rightarrow (a,b-i)$ with $1 \leq i < b$; $(a, b) \rightarrow (a-i,b-i)$ with $1 \leq i < a$. It has been proven for these
variants that the sequence $(\G(a,b))_{b \geq a}$ is ultimately additively periodic for every $a > 0$.  Now we prove this
periodicity for {\sc 2-Star Nim}. The technique used here was introduced by Landman for {Wythoff} \cite{Landman}.

\begin{theorem} \label{per.row}
The sequence $(\G(a,b))_{b \geq a}$ is ultimately additively periodic for every $a > 0$.

\end{theorem}

\begin{proof}
Set $\H(0,b) = 1$ and $\H(a,b) = \G(a,b) - b + a$ for $1 \leq a \leq b$.
By Proposition \ref{SN.2.bound}, $0 \leq \H(a,b) \leq 2a-1$.
Note that $(\G(a,b)_{b \geq 0}$ is ultimately additively periodic if and only if $(\H(a,b))_{b \geq 0}$ is ultimately
periodic. We  prove the latter by showing that $(\H(a,b))_{b \geq 0}$ can be computed on a finite state machine.

Set $M = \{\G(a-i,b), \G(a,b-j)\mid 1\leq i \leq a, 1 \leq j \leq b\}$.
By the definition of $\G$, $\G(a,b) = \mex(M)$.
Note that the number of elements in $M$ increases along with the increase of $b$. We first overcome this constraint.

Set $N_a^b = \{b-a, b-a+1, \ldots, b+a-1\}$.
By Proposition \ref{SN.2.bound}, $\G(a,b) \in N_a^b$.
Therefore, $\G(a,b) =  \mex(M) = \min(N_a^b \setminus M)$.
Note that $\G(a,b-j) < b-a$ if $j > 2a-1$ and so we can exclude these Sprague-Grundy values in $M$ when computing $\min(N_a^b
\setminus M)$.
We can skip those $b < 2a-1$ and assume that $b \geq 2a-1$.
Set $M' = \{\G(a-i,b), \G(a,b-j)\mid 1\leq i \leq a, 1 \leq j \leq 2a-1\}$.
Then $\G(a,b) = \min(N_a^b \setminus M')$.

\item [(1)] Set

$L(a,b) = \{\G(a,b-j) \mid 1 \leq j \leq 2a-1\}$,  $L^c(a,b) = N_a^b \setminus L(a,b)$,

$D(a,b) = \{\G(a-i,b) \mid 1 \leq i \leq a\}$,   \qquad   $D^c(a,b) = N_a^b \setminus D(a,b)$.
Then $\G(a,b) = \operatorname{min}(L^c(a,b) \cap D^c(a,b))$ whose $\min$ set has at most $2a$ elements.

Thus, we have shown that $\G(a,b)$ can be computed from at most $2a$ values regardless of how large $b$ is. The next step is
to show that we can store only a finite number of bits  to compute $\H(a,b)$ regardless of the increase of $b$.

\medskip

\item [(2)] Define $\S_{L^c(a,b)}$ to be the string of $2a$ bits  such that the $l^{th}$ bit $\S_{L^c(a,b)}[l]$ (with $0 \leq
    l \leq 2a-1$) is $1$ if  $b-a+l \in L^c(a,b)$ and $0$ if otherwise.
Intuitively, $\S_{L^c(a,b)}$ is obtained from the sequence $(b-a, b-a+1, \ldots, a+b-1)$ by setting the $l^{th}$ bit to $1$
if $b-a+l \in L^c(a,b)$ and $0$ if otherwise.
Equivalently, $\S_{L^c(a,b)}$ is obtained from the string of $2a$ bits  by setting the $l^{th}$ bit to $0$ if $l \in
\{\H(a,b-j) \mid 1 \leq j \leq 2a-1\}$ and $1$ if otherwise.

We next define $\S_{D^c(a,b)}$ as the string of $2a$ bits  such that the $l^{th}$ bit $\S_{D^c(a,b)}[l]$ (with $0 \leq l \leq
2a-1$) is $1$ if $b-a+l \in D^c(a,b)$ and $0$ if otherwise.
Similarly, $\S_{D^c(a,b)}$ is the string of $2a$ bits  obtained by setting the $l^{th}$ bit 
to $0$ if $l \in \{\H(a-i,b) \mid 1 \leq i \leq a\}$ and to $1$ if otherwise.

Then $\G(a,b) = (b-a) + \min\{l \mid \S_{L^c(a,b)}[l = \S_{D^c(a,b)}[l] = 1\}$ and so $\H(a,b) = \min\{l \mid
\S_{L^c(a,b)}[l] = \S_{D^c(a,b)}[l] = 1\}$.

Regardless of the increase of $b$, $\H(a,b)$ can be computed directly from two $2a$-bit sequences $\S_{L^c(a,b)}$ and
$\S_{D^c(a,b)}$.

\medskip

\item [(3)] We show that there is an algorithm that computes $\H(a,b+i)$ for all $i$ by storing only $\S_{L^c(a-i,b)}$ and
    $\S_{D^c(a-i,b)}$ for $0 \leq i \leq a$.

We first discuss how $\S_{L^c(a-i,b+1)}$ for $0 \leq i \leq a$ can be obtained from $\S_{L^c(a-i,b)}$ and $\H(a-i,b)$. By
definition,
\begin{align*}
L^c(a-i,b+1) &= N_{a-i}^{b+1} \setminus L(a-i,b+1)  \\
                 &= N_{a-i}^{b+1} \setminus \{\G(a-i,b+1-j) \mid 1 \leq j \leq 2(a-i)-1\}.
\end{align*}
Note that we exclude those $\G(a-i,b+1-j)$ with $j > 2(a-i)-1$ in the subtracted set since these values are less than
$b+1-a+i = \min(N_{a-i}^{b+1})$ by Proposition \ref{SN.2.bound}. Also note that $N_{a-i}^{b+1}$ can be obtained from
$N_{a-i}^{b}$ by removing the first (smallest) entry and adding $b+a-i$ at the end.
Equivalently, $\S_{L^c(a-i,b+1)}$ can be obtained from string $\S_{L^c(a-i,b)}$ by removing the leftmost digit and adding $1$
to the rightmost end before setting the $l$-th digit as $0$ if $\H(a-i,b) = l$, leaving other digits as they were. Note that
there is no need of storing $\H(a-i,b)$ as it can be computed from $\S_{L^c(a-i,b)}$ and $\S_{D^c(a-i,b)}$.

We next discuss how $\S_{D^c(a-i,b+1)}$ can be computed inductively on the first entry, starting from $a-i = 0$. For this
case, $\S_{D^c(0,b+1)}$ is the string of $2a$ $1$s. For $a-i > 0$, we have
\begin{align*}
D^c(a-i,b+1) &= N_{a-i}^{b+1} \setminus L(a-i,b+1)  \\
                 &= N_{a-i}^{b+1} \setminus \{\G(a-i-j,b+1) \mid 1 \leq j \leq a-i\}.
\end{align*}
Therefore, $\S_{D^c(a-i,b+1)}$ can be obtained from $\S_{D^c(a-i-1,b+1)}$ by removing the leftmost digit and adding $1$ to
the rightmost end before setting the $l$-th digit as $0$ if $\H(a-i-1,b+1) = l$, leaving other digits as they were. Note that
there is no need to store $\H(a-i-1,b+1)$ since it can be computed from $\S_{L^c(a-i-1,b+1)}$ and $\S_{D^c(a-i-1,b+1)}$.

\medskip

\item [(4)] Thus, $\H(a,b+1)$ can be computed on a finite state machine that requires storing $\S_{L^c(a-i,b)}$ and
    $\S_{D^c(a-i,b)}$ for $0 \leq i \leq a$. After that we replace stored data by increasing the second entry by 1 and
    compute $\H(a,b+2)$ and so on. This finite state machine requires \rm{O}$(a^2)$ states each of which needs at most $2a$
    bits. Using this finite state machine, we can compute the sequence $(\H(a,b))_{b \geq {2a-1}}$. Since the finite state
    machine eventually repeats, the sequence $(\H(a,b))_{b \geq 0}$ is ultimately periodic and so the sequence $(\G(a,b))_{b
    \geq 0}$ is ultimately additively periodic.
\end{proof}


The following result means every row (column) in the expanded table of Table \ref{T1} contains exactly one Sprague-Grundy
value $g$ for every $g \geq 0$.

\begin{proposition} \label{agb}
For given nonnegative integers $g$ and  $a$, there exists exactly one $b$ such that $\G(a,b)  = g$. Moreover, Proposition
\ref{SN.2.bound} implies $|b - a| \leq g$.
\end{proposition}
\begin{proof}
Let $c = a+g+1$. Consider the position $(a,c)$. The lower bound in Proposition \ref{SN.2.bound} gives $\G(a,c) \geq c-a =
g+1$, implying the existence of some $i$ such that either $\G(a-i,c) = g$ or $\G(a,c-i) = g$. The first case cannot hold as
$\G(a-i,c) \geq c-a+i = g+ 1 +i$ by Proposition \ref{SN.2.bound} and so the second case holds. Let $b = c-i$, we have $b$ as
required.
\end{proof}


\subsection{The ultimate periodicity of Sprague-Grundy values}

We end this section with a conjecture on the ultimate periodicity of the values on diagonals parallel to the main diagonal.

\begin{conjecture} \label{per.dia}
The sequence $(\G(a_1+i,a_2+i))_{i \geq 0}$ is ultimately periodic.
\end{conjecture}

As examples, the sequence $(\G(2+i,4+i))_{i \geq 0}$ appears to be ultimately periodic with pre-period length $n_0 = 8$ and
period length $p = 4$. Its first 100 values, where the first period is bold, are

\indent
5, 6, 9, 9, 3, 3, 5, 5, {\bf 3, 3, 4, 4}, 3, 3, 4, 4, 3, 3, 4, 4, 3, 3, 4, 4, 3, 3, 4, 4, 3, 3, 4, 4, 3, 3, 4, 4, 3, 3, 4, 4,
3, 3, 4, 4, 3, 3, 4, 4, 3, 3, 4, 4, 3, 3, 4, 4, 3, 3, 4, 4, 3, 3, 4, 4, 3, 3, 4, 4, 3, 3, 4, 4, 3, 3, 4, 4, 3, 3, 4, 4, 3, 3,
4, 4, 3, 3, 4, 4, 3, 3, 4, 4, 3, 3, 4, 4, 3, 3, 4, 4

\indent
The sequence $(\G(2+i,5+i))_{i \geq 0}$ appears to be ultimately periodic with pre-period length $n_0 = 28$ and period length
$p = 144$. Its first 400 values are

3, 8, 10, 10, 4, 4, 4, 19, 6, 6, 5, 5, 5, 6, 10, 6, 5, 5, 5, 6, 6, 8, 5, 5, 5, 8, 8, 11, {\bf 5, 5, 5, 8, 6, 6, 5, 5, 5, 6,
8, 6, 5, 5, 5, 6, 6, 9, 5, 5, 5, 8, 9, 8, 5, 5, 5, 9, 6, 6, 5, 5, 5, 6, 8, 6, 5, 5, 5, 6, 6, 8, 5, 5, 5, 9, 8, 8, 5, 5, 5, 8,
6, 6, 5, 5, 5, 6, 8, 6, 5, 5, 5, 6, 6, 8, 5, 5, 5, 8, 8, 10, 5, 5, 5, 8, 6, 6, 5, 5, 5, 6, 8, 6, 5, 5, 5, 6, 6, 10, 5, 5, 5,
8, 10, 8, 5, 5, 5, 10, 6, 6, 5, 5, 5, 6, 8, 6, 5, 5, 5, 6, 6, 8, 5, 5, 5, 9, 8, 8, 5, 5, 5, 8, 6, 6, 5, 5, 5, 6, 8, 6, 5, 5,
5, 6, 6, 8, 5, 5, 5, 8, 8, 10,} 5, 5, 5, 8, 6, 6, 5, 5, 5, 6, 8, 6, 5, 5, 5, 6, 6, 9, 5, 5, 5, 8, 9, 8, 5, 5, 5, 9, 6, 6, 5,
5, 5, 6, 8, 6, 5, 5, 5, 6, 6, 8, 5, 5, 5, 9, 8, 8, 5, 5, 5, 8, 6, 6, 5, 5, 5, 6, 8, 6, 5, 5, 5, 6, 6, 8, 5, 5, 5, 8, 8, 10,
5, 5, 5, 8, 6, 6, 5, 5, 5, 6, 8, 6, 5, 5, 5, 6, 6, 10, 5, 5, 5, 8, 10, 8, 5, 5, 5, 10, 6, 6, 5, 5, 5, 6, 8, 6, 5, 5, 5, 6, 6,
8, 5, 5, 5, 9, 8, 8, 5, 5, 5, 8, 6, 6, 5, 5, 5, 6, 8, 6, 5, 5, 5, 6, 6, 8, 5, 5, 5, 8, 8, 10, 5, 5, 5, 8, 6, 6, 5, 5, 5, 6,
8, 6, 5, 5, 5, 6, 6, 9, 5, 5, 5, 8, 9, 8, 5, 5, 5, 9, 6, 6, 5, 5, 5, 6, 8, 6, 5, 5, 5, 6, 6, 8, 5, 5, 5, 9, 8, 8, 5, 5, 5, 8,
6, 6, 5, 5, 5, 6, 8, 6, 5, 5, 5, 6, 6, 8, 5, 5, 5, 8, 8, 10, 5, 5, 5, 8, 6, 6, 5, 5, 5, 6, 8, 6.

\section{Further questions} \label{Ss.future}
We are interested in the following questions for further study.
\begin{enumerate} \itemsep0em
\item Formulate the Sprague-Grundy function for {\sc 2-Star Nim} and generally {\sc $m$-Star Nim}.
\item More generally, formulate the Sprague-Grundy function for {\sc Star Silver Dollar} with at most two tokens on each
    trip (not all even, not all odd).
\end{enumerate}

Having opened up the study of nearly disjunctive sums obtained by identifying the first cell of each of a number of strips
for {\sc Nim} or {\sc Silver Dollar}, it is natural to ask what happens when more cells are identified.  What if cells $0, 1,
\ldots, k$ are identified?  What can be said about computation of the Sprague-Grundy function, when $k$ is held constant?
What if the cells are identified in a different order -- for example, with two strips, cells 0 and 1 of the first strip are
identified with cells 1 and 0, respectively, of the second strip?

Given impartial games $G_1,\ldots,G_n$, their \textit{nearly disjunctive sum} is played as for the disjunctive sum,
\textit{except} that, if a player moves in $G_i$ so that the resulting position in $G_i$ is a $\P$-position for $G_i$, then
neither player can ever make any move in any $G_j$ ($j\in\{1,\ldots,n\}$) that gives a $\mathcal{P}$-position for $G_j$.


\begin{ack}
We thank Mr Kevin Bicknell at La Trobe University for his support with the diagrams of the paper.
\end{ack}


\bibliographystyle{amsplain}

\appendix

\section{Some Maple code} \label{A}

We include the Maple code for $\mex$, {\sc Nim}-sum, 2-{\sc Star Nim}, and {\sc Star Silver Dollar} of positions $([a], [b],
[c,d,e])$. The first two functions are recalled in the last two.

\subsection{``mex"} \label{A1}
\begin{verbatim}
mex := proc (S)
local i;
if S = {}
then return 0 else
    for i from 0 to max(S)+1 do
        if not i in S
            then
            return i
        end if
    end do
end if
end proc
\end{verbatim}

\subsection{{\sc Nim}-sum of two values up to 127} \label{A2}

\begin{verbatim}
ns := proc (n, m)
local a, b, c, i;
a := convert(n, binary)+convert(m, binary);
b[1] := floor((1/10000000)*a);
b[2] := floor((1/1000000)*a-10*b[1]);
b[3] := floor((1/100000)*a-100*b[1]-10*b[2]);
b[4] := floor((1/10000)*a-1000*b[1]-100*b[2]-10*b[3]);
b[5] := floor((1/1000)*a-10000*b[1]-1000*b[2]-100*b[3]-10*b[4]);
b[6] := floor((1/100)*a-100000*b[1]-10000*b[2]-1000*b[3]
        - 100*b[4]-10*b[5]);
b[7] := floor((1/10)*a-1000000*b[1]-100000*b[2]-10000*b[3]
        -1000*b[4]-100*b[5]-10*b[6]);
b[8] := a-10000000*b[1]-1000000*b[2]-100000*b[3]-10000*b[4]
        -1000*b[5]-100*b[6]-10*b[7];
for i to 8 do
    b[i] := mod(b[i], 2)
end do;
return 128*b[1]+64*b[2]+32*b[3]+16*b[4]+8*b[5]+4*b[6]+2*b[7]+b[8]

end proc
\end{verbatim}

\subsection{2-{\sc Star Nim}} \label{A3}
Texts after \# are comments.

\begin{verbatim}
SN2 := proc (a, b)
local i, j, k, g, S;

#defining values for positions with one token in zero square
for i to b do
    g[0, i] := i-1;
    g[i, 0] := i-1
end do;

#recursive calculation
for i to a do
    for j to b do
    S := {};
        for k from 0 to i-1 do         #move from one strip
            S := `union`(S, {g[k, j]})
        end do;
        for k from 0 to j-1 do         #move from one strip
            S := `union`(S, {g[i, k]})
        end do;
    g[i, j] := mex(S)
    end do
end do;

return g[a, b]

end proc
\end{verbatim}

\subsection{3-{\sc Star Silver Dollar} of position ([a], [b], [c,d,e])} \label{A4}

\begin{verbatim}

S113 := proc (alpha, beta, gamma, h, i)
local l, o, t, u, v, z, g, Omega;

#recursively assign values for positions ([0], [0], [t,u,v])
for t from 0 to gamma do
    for u from t+1 to h do
        for v from u+1 to i do
            if t = 0 then
                g(0, 0, t, u, v) := v-u-1
            else
                g(0, 0, t, u, v) := ns(t-1, v-u-1)
            end if
        end do
    end do
end do;

#recursively assign values for positions ([l], [0], [t,u,v])
for l to alpha do
    for t from 0 to gamma do
        for u from t+1 to h do
            for v from u+1 to i do
                if t = 0 then
                    g(l, 0, t, u, v) := ns(l-1, v-u-1)
                else
                    g(l, 0, t, u, v) := ns(ns(l-1, t-1), v-u-1)
                end if
            end do
        end do
    end do
end do;

#recursively assign values for positions ([0], [o], [t,u,v])
for o to beta do
    for t from 0 to gamma do
        for u from t+1 to h do
            for v from u+1 to i do
                if t = 0 then
                    g(0, o, t, u, v) := ns(o-1, v-u-1)
                else
                    g(0, o, t, u, v) := ns(ns(o-1, t-1), v-u-1)
                end if
            end do
        end do
    end do
end do;

#recursively assign values for positions ([l], [o], [0,u,v])
for l to alpha do
    for o to beta do
        for u to h do
            for v from u+1 to i do
                g(l, o, 0, u, v) := ns(ns(l-1, o-1), v-u-1)
            end do
        end do
    end do
end do;

#recursively calculate values for positions ([l], [o], [t,u,v])
for l to alpha do
    for o to beta do
        for t to gamma do
            for u from t+1 to h do
                for v from u+1 to i do
                    Omega := {};

                    #move from the strip [l]
                    for z from 0 to l-1 do
                        Omega := `union`(Omega, {g(z, o, t, u, v)})
                    end do

                    #move from the strip [o]
                    for z from 0 to o-1 do
                        Omega := `union`(Omega, {g(l, z, t, u, v)})
                    end do;

                    #move token t from the strip [t,u,v]
                    for z from 0 to t-1 do
                        Omega := `union`(Omega, {g(l, o, z, u, v)})
                    end do

                    #move token u from the strip [t,u,v]
                    if 1 < u-t then
                        for z from t+1 to z-1 do
                            Omega := `union`(Omega, {g(l, o, t, z, v)})
                        end do
                    end if

                    #move token v from the strip [t,u,v]
                    if 1 < v-u then
                        for z from u+1 to v-1 do
                            Omega := `union`(Omega, {g(l, o, t, u, z)})
                        end do
                    end if

                    g(l, o, t, u, v) := mex(Omega)
                end do
            end do
        end do
    end do
end do

return g(alpha, beta, gamma, h, i)

end proc;

\end{verbatim}

\end{document}